\documentclass[reqno,11pt]{amsart}

\usepackage{mathdots}
\usepackage{tikz}
\usepackage{tikz-cd}
\usepackage{amssymb}
\usepackage{amsgen}
\usepackage{amsmath}
\usepackage{amsthm}
\usepackage{mathrsfs}
\usepackage{cite}
\usepackage{amsfonts}
\usepackage{enumitem}

\newcommand{\End}{\mathop{\mathrm{End}}}

\newcommand{\rad}{\mathop{\mathrm{rad}}\nolimits}

\newcommand{\Ind}{\mathop{\mathrm{Ind}}\nolimits}
\newcommand{\Coind}{\mathop{\mathrm{Coind}}\nolimits}
\newcommand{\rk}{\mathop{\mathrm{rk}}\nolimits}
\newcommand{\J}{\mathrel{\mathscr J}} 
\newcommand{\R}{\mathrel{\mathscr R}} 
\newcommand{\eL}{\mathrel{\mathscr L}} 

\newcommand{\inv}{^{-1}}
\newcommand{\p}{\varphi}

\newcommand{\ov}[1]{\ensuremath{\overline {#1}}}
\newcommand{\til}[1]{\ensuremath{\widetilde {#1}}}

\newcommand{\wh}{\widehat}

\newcommand{\Hom}{\mathop{\mathrm{Hom}}\nolimits}

\newcommand{\Ext}{\mathop{\mathrm{Ext}}\nolimits}

\usepackage{xcolor}

\newtheorem{Thm}{Theorem}[section]
\newtheorem{Prop}[Thm]{Proposition}

\newtheorem{Lemma}[Thm]{Lemma}
{\theoremstyle{definition}
}
{\theoremstyle{remark}
}
\newtheorem{Cor}[Thm]{Corollary}
{\theoremstyle{remark}
}
{\theoremstyle{remark}
}

{\theoremstyle{remark}
}
{\theoremstyle{remark}
}
{\theoremstyle{remark}
}
{\theoremstyle{remark}
\newtheorem*{Claim*}{Claim}}

\numberwithin{equation}{section}

\title[Topology and monoid representations~II]{Topology and monoid representations~II: left regular bands of groups and Hsiao's monoid of ordered $G$-partitions}

\author{Benjamin Steinberg}
\address[B.~Steinberg]{%
    Department of Mathematics\\
    City College of New York\\
    Convent Avenue at 138th Street\\
    New York, New York 10031\\
    USA}
\email{bsteinberg@ccny.cuny.edu}

\thanks{The author was supported by a PSC CUNY grant, a Simons Foundation Collaboration Grant, award number 849561, and the Australian Research Council Grant DP230103184.}
\date{\today}


\begin{document}

\begin{abstract}
The goal of this paper is to use topological methods to compute $\mathrm{Ext}$ between irreducible representations of  von Neumann regular monoids in which Green's $\mathscr L$- and $\mathscr J$-relations coincide (e.g., left regular bands).   Our results subsume those of S.~Margolis, F.~Saliola, and B.~Steinberg, \emph{Combinatorial topology and the global dimension of algebras arising
  in combinatorics}, J. Eur. Math. Soc. (JEMS), \textbf{17}, 3037--3080  (2015).

   Applications include computing $\Ext$ between arbitrary simple modules and computing a quiver presentation for the algebra of Hsiao's monoid of ordered $G$-partitions (connected to the Mantaci-Reutenauer descent algebra for the wreath product $G\wr S_n$).   We show that this algebra is Koszul, compute its Koszul dual and compute minimal projective resolutions of all the simple modules using topology.  More generally, these results work for CW left regular bands of abelian groups.    These results generalize the results of S.~Margolis, F.~V. Saliola, and B.~Steinberg.
\emph{Cell complexes, poset topology and the representation theory of
  algebras arising in algebraic combinatorics and discrete geometry}, Mem. Amer. Math. Soc., \textbf{274}, 1--135, (2021).
\end{abstract}

\maketitle
\section{Introduction}
  In~\cite{rrbg}, Margolis and the author computed $\Ext^1$ between simple modules for regular left duo monoids, that is, regular monoids for which each left ideal is two-sided.  These generalize left regular bands.  The representation theory of left regular bands had been used earlier to analyze a number of important Markov chains~\cite{BHR,DiaconisBrown1,Brown1,bjorner2} and is connected with Solomon's descent algebra~\cite{SolomonDescent,BidigareThesis,Brown2,SaliolaDescent}.  Hsiao had introduced in~\cite{Hsiao} a regular left duo monoid  $\Sigma_n^G$ associated to a finite group $G$, called the monoid of ordered $G$-partitions,  such that the algebra of $S_n$-invariants of $K\Sigma_n^G$ is anti-isomorphic to the Mantaci-Reutenauer descent algebra~\cite{MantReut} of the wreath product $G\wr S_n$ of $G$ with the symmetric group (a certain subalgebra of $K[G\wr S_n]$ that serves as a noncommutative character ring).  We computed the quiver of the complex algebra of Hsiao's monoid in~\cite{rrbg}.

  In this paper, we show how to compute arbitrary $\Ext$ between simple modules of regular left duo monoids using topological means.  The author had previously shown, in joint work with Margolis and Saliola,  that $\Ext$ between simple modules for  left regular bands can be computed as the homology of associated cell complexes and projective resolutions can be built from chain complexes of contractible CW complexes with a cellular action of the monoid~\cite{MSS,ourmemoirs}). Here we do the same for regular left duo monoids.    We then apply the theory to compute arbitrary $\Ext$ between simple modules for the algebra of Hsiao's monoid.  Moreover, we compute a quiver presentation of the algebra, show it is Koszul and compute its Koszul dual as the incidence algebra of a certain poset, in the case that $G$ is abelian.  Additionally, we give an explicit construction of the minimal projective resolutions of the simple modules via the topology of permutohedra.  The results hold more generally for left regular bands of abelian groups whose idempotents are a CW left regular band in the sense of~\cite{ourmemoirs} (i.e., the set of idempotents and all its contractions are face posets of regular CW complexes).
Finally, we compute explicit projective resolutions of all simple modules from the order complex of the poset of principal right ideals for von Neumann regular left duo monoids.

\subsection*{Historical note}
As the initial version of this paper~\cite{toprep} was being polished for submission to ArXiv, the paper~\cite{2023arXiv230614985B} appeared, which has some overlap with our results.  A complete set of orthogonal primitive idempotents for regular left duo monoids (called there left regular bands of groups) is produced in~\cite{2023arXiv230614985B}, something also done in our paper.  They also produce a complete set of orthogonal primitive idempotents for the Mantaci-Reutenauer descent algebra, something we do not do in this paper.

\section{Left duo and right semicentral monoids}\label{s:secduo}
The set of idempotents of a monoid $M$ will be denoted $E(M)$.
Let $M$ be a finite monoid with group of units $G$, and let $S=M\setminus G$ be the ideal of singular elements.  Then $\mathcal EM$ denotes the classifying space~\cite{Rosenberg} of the category with objects $M$ and arrows $M\times M$ where $(m,m')\colon mm'\to m$.  The composition is given by $(m,m')(mm',m'')=(m,m'm'')$ and the identity at $m$ is $(m,1)$.  Note that $\mathcal EM$ is a contractible free $M$-CW complex (in the sense of~\cite{TopFinite1}).  Then $S\mathcal EM$ is the subcomplex of $\mathcal EM$ consisting of those cells of the form $s\sigma$ with $s\in S$ and $\sigma\in \mathcal EM$.  This is an $M$-invariant subcomplex.  See~\cite{affinetoprep} for details. If $P$ is a poset, the \emph{order complex} $\Delta(P)$ of $P$ is the simplicial complex with vertex set $P$ whose simplices are formed by the finite chains in $P$.  If $M$ acts on $P$ by order preserving maps, then $M$ acts on $\Delta(P)$ by simplicial maps.

Let us establish some notation.  If $X$ is a right $M$-set, then $\Omega_M(X) = \{xM\mid x\in X\}$ is the poset of cyclic $M$-subsets of $X$.  We sometimes write  $\Omega(X)$ if $M$ is clear from context.  Notice that if $M$ is a finite monoid with singular ideal $S$, then $M$ acts on the left of $\Omega_M(S)$  by order preserving maps.  There is an $M$-equivariant cellular mapping $\Phi_S\colon S\mathcal EM\to \Delta(\Omega_M(S))$ that is the identity on vertices.  This mapping is a homotopy equivalence when $M$ is regular (cf.~\cite[Corollary~6.18]{MSS}).

 Following established terminology in ring theory, an idempotent $e$ is \emph{right semicentral} if $eM=eMe$ and is \emph{left semicentral} if $Me=eMe$.  We denote by $\psi\colon M\to G(M)$ the group completion homomorphism.  This is the universal map from $M$ into a group.  The group $G(M)$ is called the \emph{group completion} of $M$.  When $M$ is finite, $\psi$ is a surjective homomorphism.   Concretely, it can be described in the following way using a combination of~\cite[Chapter~8, Example~1.7]{Arbib} and Graham's theorem~\cite{Graham} (see also~\cite[Chapter~4.13]{qtheor}).    Choose an idempotent $e$ in the minimal ideal of $M$.  Then $eMe=G_e$ is a group (cf.~\cite[Appendix~A]{qtheor} or~\cite{Arbib}), and $G(M)\cong G_e/N$ where $N$ is the normal closure of the subgroup of $G_e$ consisting of those elements than can be written as a product of idempotents from $MeM$.  The universal map takes $m$ to $emeN$.

In analogy to standard terminology in ring theory, we say that $M$ is \emph{left duo} if $mM\subseteq Mm$ for all $m\in M$.  This property is equivalent to $Mm=MmM$ for each $m\in M$ and is also equivalent to each left ideal of $M$ being two-sided.  For example, a band (i.e., a monoid in which each element is idempotent) is left duo if and only if it is a left regular band (i.e., satisfies the identity $xyx=xy$).  Regular left duo monoids (and their dual, right duo monoids) were studied in~\cite{rrbg} under the incorrect name of left (right) regular bands of groups (which, defined correctly, form a proper subclass).  In~\cite{rrbg} a description of the quiver of a regular left duo monoid was given, that is, a computation of $\Ext^1$ between simples was performed.  Here we give a topological description of all $\Ext$ between simples that, even for $\Ext^1$, is more conceptual than that in~\cite{rrbg}.  For left regular bands, this reduces to the results of~\cite{MSS} (see also~\cite{ourmemoirs}).

Actually, the more natural class of monoids to consider consists of those monoids whose idempotent-generated principal left ideals are two-sided, that is, monoids $M$ for which $eM\subseteq Me$ for all idempotents $e\in E(M)$ or, equivalently, monoids for which $eM=eMe$ for each idempotent $e$.  Rings with this property are sometimes called right semicentral, and so we shall call these \emph{right semicentral} monoids; the dual property is called  left semicentral.  Note that a regular monoid is right semicentral if and only if it is left duo since each principal left ideal is generated by an idempotent in a regular monoid.  The class of finite right semicentral monoids is closed under finite direct products, taking submonoids and quotient monoids.  Note that $M$ is right semicentral if and only if $eme=em$ for every idempotent $e\in E(M)$ and $m\in M$, from which these facts easily follow.

The following elementary proposition provides some structural properties of right semicentral monoids that follow from standard finite semigroup theory. We provide a proof for completeness.

\begin{Prop}\label{p:completeness}
Let $M$ be a finite right semicentral monoid.
\begin{enumerate}
\item The idempotents $E(M)$ of $M$ form a left regular band.
\item If $e,f\in E(M)$, then $Me\cap Mf=Mef$.
\item The idempotent-generated principal left ideals form a lattice $\Lambda(M)$ with intersection as the meet.
\item If $e\in E(M)$ and $m\in M$, then $e\in MmM$ if and only if $Me\subseteq Mm$, if and only if $Me\subseteq Mm^{\omega}$  where $m^{\omega}$ denotes the unique idempotent positive power of $m$.
\item There is a surjective homomorphism $\sigma_M\colon M\to \Lambda(M)$ (with respect to meet) given by $\sigma_M(m) = Mm^{\omega}$.  Moreover, $\Lambda(M)\cong \Lambda(E(M))$ and $\sigma_M|_{E(M)}$ can be identified with $\sigma_{E(M)}\colon E(M)\to \Lambda(E(M))$.
\item For each $e\in E(M)$, $eM = eMe$, and hence there is a retraction $\rho_e\colon M\to eMe$ given by $\rho_e(m)=em$.  Moreover, $eMe$ is right semicentral.
\item If $e$ is an idempotent of the minimal ideal of $M$, then $G(M)\cong eMe$ and $\rho_e$ can be identified with the universal map $\psi\colon M\to G(M)$.
\item If $e,f\in E(M)$, then $eM=fM$ if and only if $eE(M)=fE(M)$, if and only if $e=f$.
\end{enumerate}
\end{Prop}
\begin{proof}
If $e,f\in E(M)$, then since $e$ is right semicentral, we have that $efe=ef$. Thus $efef=eff=ef$ and $E(M)$ is a left regular band. This proves the first item.
The second follows because on the one hand, $Mef=Mefe\subseteq Me\cap Mf$, and on the other if $m\in Me\cap Mf$, then $m=me=mf$, and so $mef=m$.  Thus $Mef=Me\cap Mf$.  The third item follows from the second since $M= M1$ is the maximum idempotent-generated principal left ideal.

For the fourth item, trivially $Mm^{\omega}\subseteq Mm\subseteq MmM$ and so we need only show that if $e\in MmM$, then $e\in Mm^{\omega}$.  Write $e=xmy$ with $x,y\in M$.  Then $yexmyexm = yexm$, and so $f=yexm$ is an idempotent.   Now $xmfy=xmyexmy=e$, and so $e\in MfM=Mf$, as $f$ is right semicentral.  Thus $Me\subseteq Mf\subseteq Mm$. Suppose inductively that $Me\subseteq Mm^k$ with $k\geq 1$. Then, since $e$ is right semicentral, we have $e\in eMm\subseteq Mem\subseteq Mm^{k+1}$
  Thus $e\in Mm^n$ for all $n\geq 1$, and so $e\in Mm^{\omega}$, that is, $Me\subseteq Mm^{\omega}$.

For the fifth item, since $(mn)^{\omega}\in MmM\cap MnM$, we deduce from (4) that $M(mn)^{\omega}\subseteq Mm^\omega\cap Mn^\omega=Mm^{\omega}n^{\omega}$, where the last equality follows from (2).  Conversely, $Mm^{\omega}n^{\omega}\subseteq  MmnM$, and so $Mm^{\omega}n^{\omega}\subseteq M(mn)^\omega$ by (4).  Thus $M(mn)^{\omega}=Mm^{\omega}\cap Mn^{\omega}$, and so $\sigma_M$ is a homomorphism.  For the final statement of (5) it suffices to observe that for idempotents $e,f\in E(M)$, we have that $Me\subseteq Mf$ if and only if $ef=e$, if and only if $E(M)e\subseteq E(M)f$.

The first part of the sixth item follows because $e$ is right semicentral.  The monoid $eMe$ is right semicentral because if $f\in E(eMe)\subseteq E(M)$, then $fmf=fm$ for all $m\in M$ and hence, in particular, for all $m\in eMe$.

For the seventh item, it follows from standard finite semigroup theory that $eMe$ is a group~\cite[Appendix~A]{qtheor}.  Thus $\rho_e$ factors through the projection $\psi\colon M\to G(M)$.  But $\psi(em) = \psi(e)\psi(m)=\psi(m)$ since $\psi(e)=\psi(e)^2$ implies $\psi(e)=1$, and so $\psi$ factors through $\rho_e$.  Thus we can identify $\rho_e$ with $\psi$.

For the final item, note that if $eM=fM$, then $ef=f$ and $fe=e$.  But since $E(M)$ is a left regular band, $f=ef = efe=ee=e$.  This completes the proof.
\end{proof}

From now on we shall write $\sigma$ instead of $\sigma_M$.

We now turn to the case of finite regular left duo monoids.  It follows from Proposition~\ref{p:completeness}(8) that if $M$ is a regular left duo monoid with left regular band of idempotents $B$, then we can identify $B$, $\Omega_B(B)=B/{\mathscr R}$ and $\Omega_M(M)=M/{\mathscr R}$.  Moreover, $eB\subseteq fB$ if and only if $fe=e$.  Thus $B$ is a poset with respect to the ordering $e\leq f$ if $fe=e$; note that if $fe=e$ in a left regular band, then also $ef=efe=e^2=e$, and so $\leq$ is the natural partial order.  Next we want to translate the order preserving left action of $M$ (and hence its group of units) on $\Omega_M(M)$ given by $m(aM)=maM$, into an action on $B$.  If $m\in M$, then $m\in G_{m^\omega}$ by regularity.  Indeed, if $m=mam$, then $ma$ is idempotent, and direct computation yields $mam^{\omega}=m^{\omega}$.   But also $Mma\subseteq Mm^{\omega}$ by Proposition~\ref{p:completeness}(4).  Thus $m^{\omega}=mam^{\omega}=ma$. Therefore, $m^{\omega}mm^{\omega}=m^{\omega}m=mam=m$, and so $m\in G_{m^{\omega}}$ (as $m^{\omega}$ is a power of $m$).  Denote by $m^\dagger$ the inverse of $m$ in $G_{m^{\omega}}$.

\begin{Prop}\label{p:identify.poset}
Let $M$ be a finite regular left duo monoid with left regular band of idempotents $B$ and group of units $G$.
\begin{enumerate}
	\item $M$ acts by endomorphisms on $B$ via $(m,e)\mapsto mem^\dagger$ and the restriction of this action to $G$ is the action via conjugation.
	\item This action is by order preserving maps.
	\item The mapping $\lambda\colon B\to M/{\R}$ given by $\lambda(e)=eM$ is an $M$-equivariant order isomorphism.
\end{enumerate}
\end{Prop}
\begin{proof}
Note that if $g\in G$, then $g^{\omega}=1$, and so $g^\dagger=g\inv$.
We shall repeatedly use that in a regular left duo monoid, if $m,x\in M$, then $mx\in Mm=Mm^{\omega}$, and so $mxm^{\omega} =mx$.  Also $m^\dagger m = m^{\omega}=mm^\dagger$ and $m=mm^{\omega}$.  Therefore, $mem^\dagger mem^\dagger = mem^{\omega}em^\dagger=mem^\dagger$.   Also if $e,f\in B$, then $mem^\dagger mfm^\dagger = mem^{\omega}fm^\dagger = mefm^\dagger$.   Thus $e\mapsto mem^\dagger$ is an endomorphism of $B$.  If $m,n\in M$ and $e\in B$, then we note that $mne(mn)^\dagger (mn) = mne(mn)^{\omega}=mne$ and $mnen^\dagger m^\dagger mn = mnen^\dagger m^{\omega}n = mnen^\dagger n=mnen^{\omega}=mne$, and so $mne(mn)^\dagger M=mneM=mnen^\dagger m^\dagger M$, whence $mne(mn)^\dagger = m(nen^\dagger)m^\dagger$ by Proposition~\ref{p:completeness}(8).   This proves (1), and (2) follows because any endomorphism of a left regular band preserves order.

 For the third item, $\lambda$ is surjective because $M$ is regular and injective by Proposition~\ref{p:completeness}(8).
It is order preserving because $e\leq f$ implies $fe=e$ and so $eM\subseteq fM$.  Conversely, if $eM\subseteq fM$, then $fe=e$, and so $e\leq f$.  It follows that $\lambda$ is an order isomorphism.
It remains to check $M$-equivariance.  But $\lambda(mem^\dagger) = mem^\dagger M=meM$ as $mem^\dagger m=mem^{\omega}=me$.
\end{proof}

Assume that $M$ is right semicentral and finite (not necessarily regular) with left regular band of idempotents $B$. Let $G_e$ denote the group of units of $eMe=eM$ for $e\in B$.   Then we have a surjective ring homomorphism $\p_e\colon KM\to KG_e$ for any commutative ring $K$ given by
\[\p_e(m) =\begin{cases} em, & \text{if}\ em\in G_e\\ 0 , & \text{else}\end{cases}\] which is the composition of $KM\twoheadrightarrow K[eMe]\twoheadrightarrow KG_e$.  Note that $\p_e(m) =em$ if and only if $MmM\supseteq Me$ by Proposition~\ref{p:completeness}(4) and since $eM\subseteq Me$.   Thus every simple $KG_e$-module is a simple $KM$-module via inflation along $\p_e$ and it follows from the Clifford-Munn-Ponizovskii theorem~\cite[Chapter~5]{repbook} that all simple $KM$-modules arise in this way.

Let us briefly recall Green's relations~\cite{Green}; see~\cite[Appendix~A]{qtheor} or~\cite{Arbib} for details.  Elements $a,b\in M$ are $\mathscr J$-equivalent if $MaM=MbM$, $\mathscr L$-equivalent if $Ma=Mb$ and $\mathscr R$-equivalent if $aM=bM$.  A $\mathscr J$-class $J$ is called \emph{regular} if it contains an idempotent or, equivalently, each element of $J$ is regular.  The $\mathscr J$-class (respectively, $\mathscr R$-class, $\mathscr L$-class) of an element $m\in M$ is denoted by $J_m$ (respectively, $R_m$, $L_m$).  Any $\mathscr L$- or $\mathscr R$-class of a regular $\J$-class contains an idempotent.  The set of $\mathscr J$-classes is partially ordered by $J\leq J'$ if $MJM\subseteq MJ'M$.

Recall that the \emph{apex}~\cite[Chapter~5]{repbook} of a simple $KM$-module $V$ for a finite monoid $M$ is the unique minimal $\mathscr J$-class $J$ of $M$ with $JV\neq 0$ and, moreover, $J$ is a regular $\J$-class.
Under the Clifford-Munn-Ponizovskii classification, $V$ then corresponds to the simple $KG_e$-module $eV$ where $e$ is an idempotent of $J$.

If $e$ is an idempotent of a monoid $M$, then $KL_e$ has a natural $KM$-$KG_e$-bimodule structure, and is free as a right $KG_e$-module.  We put $\Ind_e(V)=KL_e\otimes_{KG_e} V$ for a $KG_e$-module $V$ and call this an \emph{induced module}.  Dually, $KR_e$ is a $KG_e$-$KM$-bimodule, and we have $\Coind_e(V)=\Hom_{KG_e}(KR_e,V)$.  If $S$ is simple, then $\Ind_e(S)$ and $\Coind_e(S)$ are indecomposable modules.  Moreover, $\Ind_e(S)$ has simple top, $\Coind_e(S)$ has simple socle and both are isomorphic to the simple module associated to $S$ in the Clifford-Munn-Ponizovskii classification.  Details can be found in~\cite[Chapter~5]{repbook}.

  For a right semicentral monoid $M$ with left regular band of idempotents $B$, the regular $\mathscr J$-classes are the same as the $\mathscr L$-classes of idempotents since $MeM=Me$ for an idempotent $e$, and so the poset of regular $\mathscr J$-classes can be identified with $\Lambda(M)\cong \Lambda(B)$ (by Proposition~\ref{p:completeness}(5))  via $J_e\mapsto Be$ for $e\in B$.  We therefore will talk about a principal left ideal of $B$ being the apex of a simple $KM$-module as shorthand for the corresponding regular $\mathscr J$-class of $M$ being the apex.

We first generalize some notation from~\cite{ourmemoirs} to right semicentral monoids; we remark that the notation in~\cite{MSS} is different. If $M$ is a finite right semicentral monoid with left regular band of idempotents $B$, then we identify $\Lambda(M)$ and $\Lambda(B)$ as per Proposition~\ref{p:completeness}(5).  If $X\in \Lambda(B)$, then $M_{\geq X}=\{m\in M\mid \sigma(m)\geq X\}$ is a submonoid by Proposition~\ref{p:completeness}(5) and is a right semicentral monoid with left regular band of idempotents $B_{\geq X}$. We call this the \emph{contraction} of $M$ to $X$ (the nomenclature comes from the theory of oriented matroids, cf.~\cite{ourmemoirs}). If $e\in B$, we put $\partial eM = eM\setminus G_e = eMe\setminus G_e$.    This is the ideal of singular elements of $eM$.  Note that $eM$ has left regular band of idempotents $eB$. The origin of this notation is that, for many naturally occurring left regular bands $B$, $eB$ is the face poset of a regular cell decomposition of a ball and $\partial eB$ is the face poset of the corresponding cellular decomposition of the boundary sphere (see~\cite{ourmemoirs}).
If $\sigma(e)\geq X$, then we can talk about $\partial eM_{\geq X}$.
If $M$ is regular, we identify $\Omega_{M_{\geq X}}(\partial eM_{\geq X})$ with $\partial eB_{\geq X}$ via Proposition~\ref{p:identify.poset} and, under this identification, the $G_e$-action corresponds to the conjugation action.

Note that if $M$ is right semicentral and $e,f\in E(M)$ with $e$ in the minimal ideal of $M$, then $fe$ is also an idempotent in the minimal ideal of $M$ with $fe\in fM=fMf$.  Thus without loss of generality we may assume $e\in fM$.

Recall that if $V$ is a finite dimensional $KG$-module for a group $V$, then $V^*$ denotes the contragredient module $\Hom_K(V,K)$ with action $(gf)(v) = f(g\inv v)$ for $g\in G$, $f\in \Hom_K(V,K)$ and $v\in V$. Note that $V^*\otimes_K W\cong \Hom_K(V,W)$ as $KG$-modules where $G$ acts on $\Hom_K(V,W)$ by conjugation.

\begin{Prop}\label{p:ext.left.duo}
Let $M$ be a finite right semicentral monoid with left regular band of idempotents $B$.  Let $e,f\in B$ be idempotents with $e$ in the minimal ideal of $M$ and $e\leq f$ (i.e., $e\in fM$). Let $K$ be a field, $V$ a simple $KG_e$-module and $W$ a simple $KG_f$-module.
\begin{enumerate}
	\item If $e=f$, then $\Ext^n_{KM}(V,W)\cong \Ext^n_{KG_e}(V,W)$. In particular, if the characteristic of $K$ does not divide $|G_e|$, then $\Ext^n_{KM}(V,W)=0$ unless $n=0$ and $V\cong W$, in which case it is isomorphic to $\End_{KG_e}(V)$.
	\item If $e\neq f$, then $\Ext^n_{KM}(W,V) =0$.
\item If $e\neq f$, $\Ext^1_{KM}(V,W)\cong \Hom_{KG_f}(\til H_{0}(\Delta(\Omega(\partial fM)), K),V^\ast\otimes_K W)$. If $M$ is regular, then   \[\Ext^1_{KM}(V,W)\cong \Hom_{KG_f}(\til H_{0}(\Delta(\partial fB), K),V^\ast\otimes_K W).\]
\item If $e\neq f$, there is a spectral sequence \[\Ext^p_{KG_f}(\til H_{q-1}((\partial fM)\mathcal E(fMf), K),V^\ast\otimes_K W)\Rightarrow_p\Ext^n_{KM}(V,W).\] If $M$ is regular, there is a spectral sequence \[\Ext^p_{KG_f}(\til H_{q-1}(\Delta(\partial fB), K),V^\ast\otimes_K W)\Rightarrow_p \Ext^n_{KM}(V,W) .\]
\item If $e\neq f$ and the characteristic of $K$ does not divide $|G_f|$, then  $\Ext^n_{KM}(V,W)\cong \Hom_{KG_f}(\til H_{n-1}((\partial fM)\mathcal E(fMf), K),V^\ast\otimes_K W)$.  If $M$ is regular, then \[\Ext^n_{KM}(V,W)\cong \Hom_{KG_f}(\til H_{n-1}(\Delta(\partial fB), K),V^\ast\otimes_K W).\]
\end{enumerate}
\end{Prop}
\begin{proof}
Assume first that $e=f$.  In particular, $eMe=G_e$. As $eV=V$ as a $KG_e$-module, applying~\cite[Corollary~2.2(2)]{affinetoprep}, we obtain $\Ext^n_{KM}(V,W)\cong \Ext^n_{KG_e}(V,W)$.  In particular, if the characteristic of $K$ does not divide $|G_e|$, then since $KG_e$ is semisimple by Maschke's theorem, $\Ext^n_{KG_e}(V,W)=0$ unless $n=0$ and, by Schur's Lemma, $V\cong W$, in which case it is isomorphic to $\End_{KG_e}(V)$.

To prove (2) we use the standard duality for finite dimensional algebras (cf.~\cite{assem} or~\cite[Remark~3.12]{affinetoprep}).  If $eM=eMe$, then $M^{op}e=eM^{op}e$ and if $e\in fM$, then $e\in M^{op}f$.  Also $G_e^{op}\cong G_e$ and $G_f^{op}\cong G_f$.  So Corollary~\cite[Corollary~2.2(1)]{affinetoprep} yields $\Ext^n_{KM}(W,V)\cong \Ext^n_{KM^{op}}(D(V),D(W))\cong \Ext^n_{K[eM^{op}e]}(D(V), D(W)e)=0$ (using right module notation) since $D(V)$ is a $K[eM^{op}e]$-module inflated to $KM$ and $D(W)e=0$ as $e$ annihilates $W$.

To prove (3)--(5), note that since $fM=fMf$, $fV=V$ and $fW=W$, we have that $\Ext^n_{KM}(V,W)\cong \Ext^n_{KfMf}(V,W)$ by Corollary~\cite[Corollary~2.2(2)]{affinetoprep}.  Thus, without loss of generality, we may assume that $f=1$.  Then, since $M$ is finite, we have  $R_1=G_1$, $1M\setminus R_1=S$ and $\Coind_1(W)=\Hom_{KG_1}(KG_1,W)\cong W$ via evaluation at $1$ as $KM$-modules.  Therefore, (3) follows from Corollary~\cite[Corollary~4.11]{affinetoprep} and (4),(5) follow from Corollary~\cite[Corollary~4.8]{affinetoprep},  where, in the regular case, we use Proposition~\ref{p:identify.poset} to identify $\Omega(\partial M)$ with $\partial B$ as $G_1$-posets.
\end{proof}

More can be said for the case of regular left duo monoids using the following result~\cite[Lemma~3.3]{rrbg} (see also~\cite[Lemma~16.6]{repbook}).

\begin{Thm}[Margolis-Steinberg]\label{t:Margolis.Steinberg}
Let $M$ be a finite regular monoid and $K$ a field.  If $I$ is an ideal of $M$ and $V,W$ are $KM$-modules annihilated by $I$, then $\Ext^n_{KM}(V,W)\cong \Ext^n_{KM/KI}(V,W)$.
\end{Thm}


Applying Theorem~\ref{t:Margolis.Steinberg} and Proposition~\ref{p:ext.left.duo}, we obtain the following computation of $\mathrm{Ext}$ between any two simple modules of a finite regular left duo monoid, generalizing the earlier results of Margolis, Saliola and the author for left regular bands~\cite{MSS,ourmemoirs}.  Note if $M$ is right semicentral with left regular band of idempotents $B$, then the simple modules with apex $Be$ are those inflated from $KG_e$.

\begin{Thm}\label{t:left.duo}
Let $M$ be a finite regular left duo monoid with left regular band of idempotents $B$ and $K$ a field.    Let $V,W$ be simple $KM$-modules with apexes $X,Y\in \Lambda(B)$, respectively.   Suppose that $Y=Bf$ with $f\in B$.  Then $\Ext^n_{KM}(V,W)=0$ unless $X\leq Y$.
\begin{enumerate}
	\item If $X=Y$, then $\Ext^n_{KM}(V,W)\cong \Ext_{KG_f}^n(V,W)$. In particular, if the characteristic of $K$ does not divide $|G_f|$, then $\Ext^n_{KM}(V,W)=0$ unless $n=0$ and $V\cong W$, in which case $\Ext^n_{KM}(V,W)\cong \End_{KG_f}(V)$.
\item If $X<Y$, then \[\Ext^1_{KM}(V,W)\cong \Hom_{KG_f}(\til H_0(\Delta(\partial fB_{\geq X}),K),V^\ast\otimes_K W).\]
\item If $X<Y$, then there is a spectral sequence \[\Ext^p_{KG_f}(\til H_{q-1}(\Delta(\partial fB_{\geq X}),K),V^\ast\otimes_K W)\Rightarrow_p \Ext^n_{KM}(V,W).\]
\item If $X<Y$ and the characteristic of $K$ does not divide $|G_f|$, then \[\Ext^n_{KM}(V,W)\cong \Hom_{KG_f}(\til H_{n-1}(\Delta(\partial fB_{\geq X}),K),V^\ast\otimes_K W).\]
\end{enumerate}
\end{Thm}
\begin{proof}
Since $M$ is regular, it follows from~\cite[Theorem~3.7]{rrbg} that if $X,Y$ are incomparable, then $\Ext^n_{KM}(V,W)=0$.	

Suppose first that $Y<X$ and let $X=Be$ with $e\in E(M)$.  Without loss of generality we may assume that $f\in eM$.  Let $I=\{m\in M\mid Y\nleq \sigma(m)\}$.  Then $I$ is a two-sided ideal of $M$ whose complement  is the submonoid $M_{\geq Y}$ of $M$.  Moreover, $V,W$ are annihilated by $I$ and hence are $KM/KI$-modules.  But $KM/KI\cong KM_{\geq Y}$ and $M_{\geq Y}f$ is the minimal ideal of $M_{\geq Y}$.  Therefore, $\Ext^n_{KM}(V,W)=0$ by Theorem~\ref{t:Margolis.Steinberg} and Proposition~\ref{p:ext.left.duo}(2).

Next assume that $X\leq Y$.  Let $I'=\{m\in M\mid X\nleq \sigma(m)\}$.  Then $I'$ is a two-sided ideal of $M$ with complement the  submonoid $M_{\geq X}$ of $M$.  Furthermore, $V,W$ are annihilated by $I'$ and hence are $KM/KI'$-modules.  Like before, $KM/KI'\cong KM_{\geq X}$ and the minimal ideal of $M_{\geq X}$ is generated by any idempotent $e$ with $X=Be$.  Applying Theorem~\ref{t:Margolis.Steinberg} and Proposition~\ref{p:ext.left.duo}(1) and (3)--(5) yields the result.
\end{proof}

Theorem~\cite[Theorem~5.15]{rrbg} (suitably dualized) computes the same $\Ext$ spaces as Theorem~\ref{t:left.duo}(2). Although it is not too difficult to verify directly that both theorems gives the same answer, the statement in Theorem~\ref{t:left.duo}(2) is much cleaner and easier to work with. The quivers of a number of natural regular left duo monoids were computed by Margolis and the author in~\cite{rrbg}, including that of a regular left duo monoid introduced by Hsiao~\cite{Hsiao} to study the Mantaci-Reutenauer descent algebra~\cite{MantReut} for Type B Coxeter groups and other wreath products with symmetric groups.  We shall use our results to compute all $\mathrm{Ext}$ between simple modules for Hsiao's monoid and a quiver presentation for its algebra in the next section.

We end this section by constructing a complete set of orthogonal primitive idempotents for the algebra of a regular left duo monoid and describing the projective indecomposable modules.  These results were independently obtained in~\cite{2023arXiv230614985B}, as discussed in the historical note.

Fix for the remainder of this section a regular left duo monoid $M$ with left regular band of idempotents $B$. Let $K$ be a field.  We recall the construction of a complete set of orthogonal primitive idempotents for $KB$ from~\cite{Saliola}.  These idempotents will no longer be primitive in $KM$, but we shall use them to construct the primitive idempotents.

Fix for each $X\in \Lambda(B)$ an element $e_X$ with $X=Be_X$.  Define $\eta_X$
recursively by
\begin{equation}\label{eq:defidempotents}
\eta_X =e_X-\sum_{Y<X}e_X\eta_Y.
\end{equation}
Observe that, by induction, one can write
\begin{align*}
    \eta_X=\sum_{b\in B} c_bb
\end{align*}
with the $c_b$ integers such that $e_X\geq b$ for all $b$ with $c_b\neq 0$ and
the coefficient of $e_X$ in $\eta_X$ is $1$.  Also note that by construction \[e_X=\sum_{Y\leq X}e_X\eta_Y.\]  In particular, since $e_{B}=1$, we have that $1=\sum_{X\in \Lambda(B)}\eta_X$.  Also, by construction, $e_X\eta_X=\eta_X=\eta_Xe_X$.

We refer the reader to~\cite[Theorem~4.2]{ourmemoirs} for a proof of the following.

\begin{Prop}\label{p:primidempotentprops}
The elements $\{\eta_X\}_{X\in \Lambda(B)}$ have the following properties.
\begin{enumerate}
\item If $a\in B$ and $X\in \Lambda(B)$ are such that $a\notin B_{\geq X}$, then
    $a\eta_X=0$.
\item The collection $\{\eta_X\mid X\in \Lambda(B)\}$ is a complete set of orthogonal primitive idempotents of $KB$.
\end{enumerate}
\end{Prop}

An immediate consequence of Proposition~\ref{p:primidempotentprops} is the following.

\begin{Cor}\label{c:annihilate.not.above}
Let $X\in \Lambda(B)$ and suppose that $m\in M\setminus M_{\geq X}$.  Then $m\eta_X=0$.
\end{Cor}
\begin{proof}
Note that $m\eta_X=mm^{\omega}\eta_X =0$ by Proposition~\ref{p:primidempotentprops}, as $m^{\omega}\notin B_{\geq X}$.
\end{proof}

We now show that $KM\eta_X\cong KL_{e_X}$.

\begin{Prop}\label{p:schutz.proj}
Let $M$ be a regular left duo monoid with left regular band of idempotents $B$.   Let $\eta_X$ be as in \eqref{eq:defidempotents} and $I_X=M\setminus M_{\geq X}$.
\begin{enumerate}
  \item There is an isomorphism $KG_{e_X}\to \eta_XKM\eta_X$ given by $g\mapsto \eta_Xg\eta_X=g\eta_X$ with inverse $\eta_Xa\eta_X\mapsto \eta_Xa\eta_X+KI_X$ (identifying $KG_{e_X}$ with $e_X(KM/KI_X)e_X$).
  \item $KM\eta_X\cong KL_{e_X}$ as a $KM$-$KG_{e_X}$-bimodule via the above identification of $\eta_XKM\eta_X$ with $KG_{e_X}$.
\end{enumerate}
In particular, $KL_{e_X}$ is a projective module.
\end{Prop}
\begin{proof}
First note that $\eta_X+KI_X=e_X+KI_X$.  It follows that if $g\in G_{e_X}$, then $\eta_Xg-g=\eta_Xg-e_Xg\in KI_X$.  Since $KI_X\eta_X=0$ by Corollary~\ref{c:annihilate.not.above}, we deduce that $\eta_Xg\eta_X=g\eta_X$.  It then follows that $g\mapsto \eta_Xg\eta_X$ induces a homomorphism $\rho\colon KG_{e_X}\to \eta_XKM\eta_X$.  Moreover, $\eta_XKM\eta_X\cong \eta_X(KM/KI_X)\eta_X=e_X(KM/KI_X)e_X=KG_{e_X}$ where the isomorphism follows from Corollary~\ref{c:annihilate.not.above}.   The composition of $\rho$ with the isomorphism $\eta_XKM\eta_X\cong KG_{e_X}$ in the previous sentence is the identity map, from which the first item follows.

 If $m\in M_{\geq X}$, then  $m\eta_X=me_X\eta_X$  and $me_X\in L_{e_X}$.   Therefore, in light of Corollary~\ref{c:annihilate.not.above},  $KM\eta_X$ is spanned over $K$ by the elements $m\eta_X$ with $m\in L_{e_X}$. Note that $KL_{e_X}\cong (KM/KI_X)e_X$ as $KM$-$KG_{e_X}$-bimodules. If $\pi\colon KM\to KM/KI_X$ is the projection, then $\pi(m\eta_X)=me_X$ for $m\in L_{e_X}$ as $\eta_X+KI_X=e_X+KI_X$.  It follows that $\pi$ takes $KM\eta_X$ isomorphically to $(KM/KI)e_X\cong KL_{e_X}$ as a $KM$-module.  Moreover, $\pi(m\eta_Xg\eta_X) = mg+KI_X$ for $g\in G_{e_X}$, and hence $\pi$ is a bimodule isomorphism, as required.
\end{proof}

Our next goal is to show that if $S$ is a simple $KG_{e_X}$-module with projective cover $P$, then $\Ind_{e_X}(P)$ is a projective indecomposable $KM$-module and is the projective cover of $S$ (viewed as a simple $KM$-module).

\begin{Cor}\label{c:proj.cover}
Let $M$ be a regular left duo monoid and $K$ a field.  Let $e\in E(M)$ be an idempotent and $S$ a simple $KG_e$-module with projective cover $P$.  Then $\Ind_e(P)$ is a projective indecomposable $KM$-module and is the projective cover of $S$, inflated to a $KM$-module via $\p_e$.
\end{Cor}
\begin{proof}
It is a general fact that $\Ind_e$ preserves indecomposablity, cf.~\cite[Corollary~4.10]{repbook}.  Since $KL_e\cong KL_e\otimes_{KG_e} KG_e= \Ind_e(KG_e)$ and $P$ is a direct summand in $KG_e$, we deduce that $\Ind_e(P)$ is a direct summand in $KL_e$.  But $KL_e$ is projective by Proposition~\ref{p:schutz.proj}.  Thus $\Ind_e(P)$ is a projective indecomposable module and hence has simple top.  Composing the surjective homomorphisms $\Ind_e(P)\to \Ind_e(S)$ with the natural map $\Ind_e(S)\to \Coind_e(S)\cong S$ (whose kernel is $\rad(\Ind_e(S))$ by~\cite[Theorem~5.5]{repbook}), we deduce that there is a surjective homomorphism $\Ind_e(P)\to S$, and so $\Ind_e(P)$ is the projective cover of $S$.
\end{proof}

We are now prepared to construct a complete set of primitive orthogonal idempotents for $KM$ when $M$ is regular left duo (assuming we know how to do this for its maximal subgroups).

\begin{Thm}\label{t:complete.set.left.duo.reg}
Let $M$ be a finite regular left duo monoid with left regular band of idempotents $B$ and $K$ a field.  Let $\{\eta_X\mid X\in \Lambda(B)\}$ be as per \eqref{eq:defidempotents}.  For each $X\in \Lambda(B)$, choose a complete set of orthogonal primitive idempotents $\theta_{X,1},\ldots, \theta_{X,k_X}$ summing up to $e_X$ in $KG_{e_X}\subseteq KM$.  Then the elements $\gamma_{X,i}=\eta_X\theta_{X,i}\eta_X$ of $KM$ form a complete set of orthogonal primitive idempotents of $KM$ with $KM\gamma_{X,i}$ the projective cover of $KG_{e_X}\theta_{X,i}/\rad(KG_{e_X})\theta_{X,i}$, inflated to $KM$.
\end{Thm}
\begin{proof}
First note that since $\sum_{i=1}^{k_X} \theta_{X,i}=e_X$, we have $\sum_{X\in \Lambda(B)}\sum_{i=1}^{k_X}\gamma_{X,i} = \sum_{X\in \Lambda(B)}\eta_Xe_X\eta_X =\sum_{X\in \Lambda(B)}\eta_X=1$.

Next we prove that the $\gamma_{X,i}$ form an orthogonal set of idempotents.
Clearly, if $X\neq Y$, then $\gamma_{X,i}\gamma_{Y,j} =\eta_X\theta_{X,i}\eta_X\eta_Y\theta_{Y,j}\eta_Y=0$.  So it remains to show that $\gamma_{X,i}\gamma_{X,j}=\delta_{ij}\gamma_{X_i}$.  Note that $\gamma_{X,i}\gamma_{X,j}=\eta_X\theta_{X,i}\eta_X\theta_{X,j}\eta_X=\eta_X\theta_{X,i}\theta_{X,j}\eta_X$ by Proposition~\ref{p:schutz.proj}(1).  We conclude that $\eta_X\theta_{X,i}\eta_X\theta_{X,j}\eta_X=\delta_{ij}\eta_X\theta_{X,i}\eta_X$.

To complete the proof, it suffices to show that $KM\gamma_{X,i}\cong \Ind_{e_X}(KG_{e_X}\theta_{X,i})$ by Corollary~\ref{c:proj.cover}.  By Proposition~\ref{p:schutz.proj}, we have a bimodule isomorphism $KM\eta_X\to KL_{e_X}$.  Let $I_X=M\setminus M_{\geq X}$.  The proof shows that the isomorphism is induced by the projection $\pi\colon KM\to KM/KI_X$, which sends $KM\eta_X$ to $(KM/KI_X)\eta_X=(KM/KI_X)e_X\cong KL_{e_X}$. Now $\pi$ restricts to an isomorphism of $KM\gamma_{X,i}\leq KM\eta_X$ with its image.  But $\pi(\gamma_{X,i}) = \theta_{X,i}+KI_X$ as $\pi(\eta_X)=e_X+KI_X$.  Therefore, $\pi(KM\gamma_{X,i})=(KM/KI_X)e_X\theta_{X,i}$.  But $\Ind_{e_X}(KG_{e_X}\theta_{X,i})=KL_{e_X}\otimes_{KG_{e_X}}KG_{e_X}\theta_{X,i}\cong (KM/KI_X)e_X\otimes_{KG_{e_X}}KG_{e_X}\theta_{X,i}\cong (KM/KI_X)e_X\theta_{X,i}$ where the last isomorphism is via the multiplication map.  This completes the proof.
\end{proof}

\section{Applications to Hsiao's monoid model for the Mantaci-Reutenauer descent algebra}\label{s:Hsiao}

We now apply our results to compute $\mathrm{Ext}$ between simple modules for the algebra of Hsiao's monoid of ordered $G$-partitions~\cite{Hsiao} for a finite group $G$.  When $G$ is abelian, we are able to compute a quiver presentation for the algebra over algebraically closed fields of good characteristic, prove it is Koszul and show that its Koszul dual is the incidence algebra of a certain poset.  This generalizes the case when $G$ is trivial~\cite{Saliolahyperplane}.  The reader is referred to~\cite[Chapter~4.4]{ourmemoirs} or~\cite{BGSKoszul} for the definitions of a Koszul algebra and its Koszul dual.

\subsection{Quivers}
We begin this section by recalling some background on quivers and finite dimensional algebras.  The reader is referred to~\cite{assem,benson} for details.
Recall that a finite dimensional $K$-algebra $A$ is \emph{split basic} if $A/\rad(A)\cong K^n$ for some $n\geq 1$, i.e., each simple $A$-module is one-dimensional.  For example, the  algebra of a finite group over an algebraically closed field whose characteristic does not divide its order is split basic if and only if it is abelian.  If $M$ is right semicentral and  $K$ is an algebraically closed field whose characteristic does not divide the order of any maximal subgroup of $M$, then each simple module is inflated from the group algebra of a maximal subgroup, and hence $KM$ is split basic if and only if each maximal subgroup of $M$ is abelian.

Every split basic algebra is isomorphic to $KQ/I$ where $Q$ is a finite quiver and $I$ is an admissible ideal.
A \emph{quiver} $Q$ is a directed multigraph (possibly with loops).  All quivers shall be assumed finite. The vertex set of $Q$ is denoted $Q_0$ and the edge set $Q_1$. More generally, $Q_n$ will denote the set of (directed) paths of length $n$ with $n\geq 0$ (where the distinction between empty paths and vertices is ignored). Denote by $Q^{op}$ the quiver obtained from $Q$ by reversing the orientation of each arrow.

The \emph{path algebra} $KQ$ of $Q$ has a basis consisting of all paths in $Q$ (including an empty path $\varepsilon_v$ at each vertex) with the product induced by concatenation (where undefined concatenations are set equal to zero).  Here we compose paths from right to left: so if $p\colon v\to w$ and $q\colon w\to z$ are paths, then their composition is denoted $qp$.   The identity of $K Q$ is the sum of all empty paths. Note that $KQ=\bigoplus_{n\geq 0} KQ_n$, and this provides a grading of $KQ$.

The arrow ideal $J$ is the ideal of $KQ$ generated by all edges.   So $J=\bigoplus_{n\geq 1} KQ_n$ is a graded ideal and $KQ/J\cong K^{Q_0}$ with the grading concentrated in degree $0$.  An ideal $I$ of $KQ$ is called \emph{admissible} if there exists $n\geq 2$ with $J^n\subseteq I\subseteq J^2$.   A pair $(Q,I)$ consisting of a quiver $Q$ and an admissible ideal $I$ of $KQ$ is called a \emph{bound quiver}.  If $(Q,I)$ is a bound quiver, then $KQ/I$ is a split basic finite dimensional algebra with Jacobson radical $J/I$. If the ideal $I$ is homogeneous, then $KQ/I$ is graded.   Conversely, if $A$ is a split basic $K$-algebra, there is a unique (up-to-isomorphism) quiver $Q(A)$ such that $A\cong KQ(A)/I$ for some admissible ideal $I$ (the ideal $I$ is not unique).

The vertices of $Q(A)$ are the (isomorphism classes of) simple $A$-modules. If $S,S'$ are simple $A$-modules, then the number of directed edges $S\to S'$ in $Q(A)$ is $\dim_K\Ext^1_A(S,S')$.  To explain how the admissible ideal $I$ is obtained, fix a complete set of orthogonal primitive idempotents $E$ for $A$, and let $e_S\in E$ be the idempotent with $S\cong (A/\rad(A))e_S$. One has that \[\Ext^1_A(S,S')\cong e_{S'}[\rad(A)/\rad^2(A)]e_S\] as $K$-vector spaces~\cite{assem,benson}.  One can define a homomorphism $KQ(A)\to A$  by sending the empty path at $S$ to $e_S$ and by choosing a bijection of the set of  arrows $S\to S'$ with a subset of $e_{S'}\rad(A)e_S$ mapping to a basis of $e_{S'}[\rad(A)/\rad^2(A)]e_S$. The map is extended to paths of length greater than one in the obvious way.  Such a homomorphism is automatically surjective and the kernel is an admissible ideal. We remark that if there is no path of length greater than one from $S$ to $S'$ in $Q(A)$, then $e_{S'}\rad^2(A)e_S=0$ because $\rad(A)=J/I$. Similarly, if there is no path at all from $S$ to $S'$ through $S''$, then $e_{S'}\rad(A)e_{S''}\rad(A)e_S=e_{S'}(J/I)e_{S''}(J/I)e_S=0$.  Finally, we remark that if $S\neq S'$, then $e_{S'}Ae_S=e_{S'}\rad(A)e_S$~\cite{assem,benson}.   Note that if $(Q,I)$ is a bound quiver, then the quiver of $KQ/I$ is isomorphic to $Q$ with the vertex $v$ corresponding to the simple module $[(KQ/I)/(J/I)]\varepsilon_v$.

A \emph{system of relations} for an admissible ideal $I$ is a finite subset \[R\subseteq \bigcup_{v,w\in Q_0} \varepsilon_wI\varepsilon_v\] generating $I$ as an ideal. Every admissible ideal admits a system of relations $R$ and the pair $(Q,R)$ is called a \emph{quiver presentation} of $A$.  A system of relations $R$ is called \emph{minimal} if no proper  subset of $R$ generates the ideal $I$.   In the case that the quiver $Q$ is acyclic, one can use homological means to compute the size of a minimal system.     The following result combines~\cite[1.1 and 1.2]{Bongartz}. A detailed exposition can be found in~\cite[Theorem~4.1]{ourmemoirs}.

\begin{Thm}[Bongartz]\label{t:bongartz}
Let $(Q,I)$ be a bound quiver with $Q$ acyclic and let $A=KQ/I$.  For $v\in Q_0$, let $S_v$ be the simple $A$-module $(A/\rad(A))\varepsilon_v$.  Then, for any minimal system of relations $R$, one has that the cardinality of $R\cap \varepsilon_wI\varepsilon_v$ is $\dim_K\Ext^2_A(S_v,S_w)$.
\end{Thm}

\subsection{Quivers and $\Ext$ for left regular bands of groups}

A \emph{left regular band of groups} is a regular left duo monoid  $M$ in which Green's relation $\mathscr R$ is a congruence, that is, $aM=bM$ implies $acM=bcM$ for all $a,b,c\in M$.  In this case, $M$ has a split homomorphism to a left regular band such that the inverse image of each element is a group.  In~\cite{rrbg}, all regular left duo monoids were called mistakenly called left regular bands of groups.

Recall that if $M$ is a finite regular left duo monoid and $m\in M$, then $m\in G_{m^{\omega}}$ and $m^{\dagger}$ denotes the inverse of $m$ in $G_{m^{\omega}}$.

\begin{Prop}\label{p:lrbg.pi}
Let $M$ be a finite regular left duo monoid with left regular band of idempotents $B$.  Then the following are equivalent.
\begin{enumerate}
\item  $M$ is a left regular band of groups.
\item $(mn)^{\omega}=m^{\omega}n^{\omega}$ for all $m,n\in M$.
\item  $mem^{\dagger}=m^{\omega}e$ for all $m\in M$ and $e\in B$.
\end{enumerate}
\end{Prop}
\begin{proof}
If $\mathscr R$ is a congruence, then $mM=m^{\omega}M$ implies $(mn)^{\omega}M=mnM=m^{\omega}nM=m^{\omega}n^{\omega}M$, and so $(mn)^{\omega}=m^{\omega}n^{\omega}$ by Proposition~\ref{p:completeness}(8). Therefore, (1) implies (2).  Assume that (2) holds.  Then $meM=(me)^{\omega}M = m^{\omega}eM$, and so $mem^\dagger = m^{\omega}e$ by Proposition~\ref{p:identify.poset}, yielding (3).  Finally, suppose (3) holds,  and let $aM=bM$ with $a,b\in M$ and $m\in M$.  Then $a^{\omega}M=aM=bM=b^{\omega}M$, and so $a^{\omega}=b^{\omega}$ by Proposition~\ref{p:completeness}(8).  Putting $e=m^{\omega}$, we have $aea^\dagger = a^{\omega}e=b^{\omega}e=beb^\dagger$ by assumption, and so $amM=a(eM)=b(eM)=bmM$ by Proposition~\ref{p:identify.poset}.
\end{proof}

It follows easily from Proposition~\ref{p:lrbg.pi} that the class of finite left regular bands of groups is closed under finite direct products, submonoids and homomorphic images as any homomorphism of finite monoids preserves $()^{\omega}$.   In particular, any submonoid of a direct product of a finite left regular band and a finite group is a left regular band of groups.  Hsiao's monoid~\cite{Hsiao} is an example of such.

 We now compute $\Ext$ between simple modules for a left regular band of groups over fields of good characteristic.  Put $\til\beta_q(X,K)=\dim_K \til H_q(X,K)$ for a finite simplicial complex $X$.

\begin{Thm}\label{t:lrbg.ext}
Let $M$ be a finite left regular band of groups with left regular band of idempotents $B$ and $K$ a field whose characteristic divides the order of no maximal subgroup of $M$.    Let $V,W$ be simple $KM$-modules with apexes $X,Y\in \Lambda(B)$, respectively.   Suppose that $Y=Bf$ with $f\in B$.  Then $\Ext^n_{KM}(V,W)=0$ unless $X\leq Y$.
\begin{enumerate}
	\item If $X=Y$, then $\Ext^n_{KM}(V,W)=0$ unless $n=0$ and $V\cong W$, in which case $\Ext^n_{KM}(V,W)\cong \End_{KG_f}(V)$.
\item If $X<Y$, then $\Ext^n_{KM}(V,W)\cong \Hom_{KG_f}(V,W)^{\til \beta_{n-1}(\Delta(\partial fB_{\geq X}),K)}$.
\end{enumerate}
\end{Thm}
\begin{proof}
The conjugation action of $G_f$ of $fB_{\geq X}$ is trivial by Proposition~\ref{p:lrbg.pi} since $geg\inv =geg^\dagger = g^{\omega}e=fe=e$ for all $g\in G_f$ and $e\in fB_{\geq X}$.  The result then follows immediately from Theorem~\ref{t:left.duo}, as $\Hom_{KG_f}(K,V^\ast\otimes_K W)\cong \Hom_{KG_f}(K,\Hom_K(V,W))\cong \Hom_{KG_f}(V,W)$.
\end{proof}

We now describe the quiver for the algebra of a left regular band  of abelian groups $M$ (i.e., a left regular band of groups with abelian maximal subgroups) over an algebraically closed field $K$ whose characteristic does not divide the order of any maximal subgroup of $M$.  These have split basic algebras over $K$.  Fix, for each $X\in \Lambda(B)$, an idempotent $e_X\in B$ with $Be_X=X$, where, as usual, $B$ denotes the left regular band of idempotents of $M$.  Let $G_X$ denote the maximal subgroup $G_{e_X}$ and put $\wh G_X=\Hom(G_X,K^\times)$, which is the dual group of $G$.  Then simple $KM$-modules are indexed by pairs $(X,\chi)$ with $\chi\in \wh G_X$ (this representation inflates $\chi$ to $M$ via the projection $\p_{e_X}\colon KM\to KG_X$).  Recall that if $X\in \Lambda(B)$, we have a retraction $\rho_X\colon M_{\geq X}\to G_X$ given by $\rho_X(m) = e_Xm$, which restricts to a group homomorphism $G_Y\to G_X$ for $X\leq Y$.  Let's define a partial order on the set $P(M)$ of such pairs by putting $(X,\chi)\leq (Y,\theta)$ if $X\leq Y$ and $\theta=\chi\circ\rho_X$.  This relation is clearly a reflexive and antisymmetric relation as $\rho_X$ is the identity on $G_X$.    Note that if $X\leq Y\leq Z$ and $g\in G_Z$, then $\rho_X\rho_Y(g) = e_Xe_Yg=e_Xg=\rho_X(g)$ since $X\leq Y$.  Thus, this yields transitivity as $(X,\chi)\leq (Y,\theta)\leq (Z,\tau)$ implies that $\chi\circ \rho_X|_{G_Z} = \chi\circ \rho_X|_{G_Y}\circ \rho_Y|_{G_Z} = \theta\circ \rho_Y|_{G_Z}=\tau$.   We can now describe the quiver of $KM$ in terms of this partial order.  This is a special case of a result of~\cite{rrbg}, but formulated in a cleaner way.

\begin{Cor}\label{c:quiver.lrbg}
Let $M$ be a finite left regular band of abelian groups with left regular band of idempotents $B$ and $K$ an algebraically closed field whose characteristic divides the order of no maximal subgroup of $M$.  Then there are no edges from $(X,\chi)$ to $(Y,\theta)$, unless $(X,\chi)<(Y,\theta)$, in which case there is one fewer edge from $(X,\chi)$ to $(Y,\theta)$ than connected components of $\partial e_YB_{\geq X}$.
\end{Cor}
\begin{proof}
This is almost immediate from Theorem~\ref{t:lrbg.ext} once we observe that $\chi$ inflates to $\chi\circ\rho_X$ as a representation of $G_Y$, and  $\Hom_{KG_Y}(\chi\circ\rho_X,\theta)=0$ unless $\chi\circ \rho_X=\theta$, in which case it is one-dimensional.
\end{proof}

The following observation will play a crucial role in understanding how the poset $P(M)$ relates to $\Lambda(B)$.  If $P$ is a poset and $p\in P$, we put $P_{\geq p}=\{q\in P\mid q\geq p\}$.

\begin{Prop}\label{p:look.up}
Let $M$ be a finite left regular band of abelian groups with left regular band of idempotents $B$ and $K$ an algebraically closed field whose characteristic does not divide the order of any maximal subgroup of $M$.  Consider the poset $P(M)$ as defined above, and let $(X,\chi)\in P(M)$.  Then $P(M)_{\geq (X,\chi)}\cong \Lambda(B)_{\geq X}$ as posets.
\end{Prop}
\begin{proof}
Clearly the projection to the first coordinate is an order-preserving map $P(M)_{\geq (X,\chi)}\to \Lambda(B)_{\geq X}$.  The inverse order-preserving map takes $Y\geq X$ to $(Y,\chi\circ \rho_X)$.  This is order preserving since if $Y\leq Z$, then $\chi\circ \rho_X|_{G_Y}\circ\rho_Y|_{G_Z} =\chi\circ\rho_X|_{G_Z}$.
\end{proof}

Recall that a poset is \emph{ranked} or \emph{graded}, if given $p<q$, all maximal chains between $p$ and $q$ have the same length.  This means that the order complex of the open interval $(p,q)$ is a pure simplicial complex (that is, all maximal faces have the same dimension).  We define the \emph{rank} $\rk[p,q]$  of the closed interval $[p,q]$ to be the length of all maximal chains between $p$ and $q$, that is, the dimension of the order complex of $[p,q]$.  It follows from Proposition~\ref{p:look.up} that each closed (respectively, open) interval in $P(M)$ is isomorphic to a closed (respectively, open) interval in $\Lambda(B)$.  Hence, if $\Lambda(B)$ is graded, then $P(M)$ is also graded.  If each open interval of $\Lambda(B)$ is Cohen-Macaulay (see~\cite{topmethods} for the definition), then the same is true for $P(M)$.

\subsection{Quivers and $\Ext$ for CW left regular bands of groups and Hsiao's monoid}
  We  recall the hyperplane face monoid associated to the braid arrangement, using a combinatorial description due to Bidigare~\cite{BidigareThesis}.  Details can also be found in~\cite{Brown2} and~\cite[Chapter~14]{repbook}.  Let $\Sigma_n$ denote the set of all ordered set partitions of $\{1,\ldots, n\}$.  Then $\Sigma_n$ has the structure of a left regular band with respect to the multiplication
\[(P_1,\ldots, P_r)(Q_1,\ldots Q_s) = (P_1\cap Q_1,\ldots, P_1\cap Q_s, P_2\cap Q_1,\ldots, P_r\cap Q_s)^{\wedge}\]
where $\wedge$ means to omit empty intersections.  The identity is the ordered set partition with one block.
Two ordered set partitions generate the same left ideal if and only if they have the same underlying set partition and, in fact, $\Lambda(\Sigma_n)$ is isomorphic to the lattice $\Pi_n$ of set partitions ordered so that coarser partitions are larger.  We make this identification from now on and write $\sigma(\pi)$ for the underlying set partition of $\pi\in \Sigma_n$.

If $X<Y$ in $\Pi_n$ and $\pi\in \Sigma_n$ with $\sigma(\pi)=Y$, then $\Delta(\partial \pi(\Sigma_n)_{\geq X})$ is known to be a triangulation of an $(|X|-|Y|-1)$-sphere, where if $Z$ is a set partition, then $|Z|$ is the number of blocks of $Z$; see~\cite[Section~4.8]{MSS} and~\cite[Proposition~3.16]{ourmemoirs} for details.

There is a natural action of the symmetric group $S_n$ on $\Sigma_n$ by automorphisms and Bidigare~\cite{BidigareThesis} observed that $(K\Sigma_n)^{S_n}$ is anti-isomorphic to Solomon's descent algebra~\cite{SolomonDescent} for $S_n$.  Hsiao~\cite{Hsiao} generalized this approach to realize the Mantaci-Reutenauer descent algebra~\cite{MantReut} for a wreath product $G\wr S_n$, with $G$ a finite group, as an algebra of invariants of a monoid algebra.

So let $G$ be a finite group.  Then Hsiao's monoid $\Sigma_n^G$ of \emph{ordered $G$-partitions} is the submonoid of $\Sigma_n\times G^n$ consisting of all pairs $(\pi, g)$ with $\pi\in \Sigma_n$ and $g\in G^n$ such that $g_i=g_j$ whenever $i,j$ are in the same block of $\sigma(\pi)$.  Hsiao's monoid is described in a slightly different way in~\cite{Hsiao} and~\cite{rrbg}, but it is easy to see that this description gives an isomorphic monoid.  As a submonoid of a direct product of a finite left regular band with a finite group, we see that $\Sigma_n^G$ is a left regular band of groups.  Hsiao~\cite{Hsiao} showed that there is a natural action of $S_n$ by automorphisms on $\Sigma_n^G$ and that $(K\Sigma_n^G)^{S_n}$ is anti-isomorphic to the Mantaci-Reutenauer descent algebra for $G\wr S_n$~\cite{MantReut}.  Saliola~\cite{SaliolaDescent} used the description~\cite{BidigareThesis} of Solomon's descent algebra~\cite{SolomonDescent} as the invariant subalgebra $(K\Sigma_n)^{S_n}$ to make progress on understanding the descent algebra, in part because he obtained a quiver presentation of $K\Sigma_n$~\cite{Saliolahyperplane}.   For $G$ abelian and $K$ algebraically closed of characteristic not dividing $|G|$, we shall compute a quiver presentation for $K\Sigma_n^G$, generalizing that of $K\Sigma_n$.  One can then hope to apply Saliola's program to obtain a better understanding of the Mantaci-Reutenauer descent algebra.

If $X\in \Pi_n$, let $G_X$ be the subgroup of $G^n$ consisting of those elements such that $g_i=g_j$ whenever $i,j$ are in the same block of $X$; of course, $G_X\cong G^{|X|}$.  Identifying $\Sigma_n$ with $\Sigma_n\times \{1\}$,  if $\sigma(\pi)=X$, then $G_{\pi} = \{\pi\}\times G_X$; hence we can identify $G_{\pi}$ with $G_X$ canonically via the projection to the second coordinate, which we do from now on.  Also note that if $X\leq Y$, then $G_Y\leq G_X$.

The key features of the representation theory of $\Sigma_n^G$ apply to a more general class of left regular bands of groups.  We first recall the notion of a CW left regular band from~\cite{ourmemoirs}.  A finite poset $P$ is a \emph{CW poset} if it is the face poset of a regular CW complex (i.e., the set of open cells ordered by $\sigma\leq \tau$ if $\sigma \subseteq \ov \tau$). CW posets can be axiomatized~\cite{BjornerCW} as graded posets $P$ such that $\Delta(P_{<p})$ is a triangulation of sphere of dimension $\dim(\Delta(P_{<p}))$,  where $P_{<p} = \{q\in P\mid q<p\}$.  Note that $\Delta(P_{\leq p})$ is then a triangulation of a ball since it is a cone on $\Delta(P_{<p})$, where $P_{\leq p}=\{q\in P\mid q\leq p\}$.
A left regular band $B$ is a CW left regular band if $B_{\geq X}$ is a CW poset for all $X\in \Lambda(B)$. If $B$ is a monoid, as will be the case in this paper, we must then have $B_{\geq X}$ is the face poset of a regular CW decomposition of a ball (since $B_{\geq X}$ has a maximum element $1$).  In general, $eB_{\geq X}$ is the face poset a regular CW decomposition of a ball and $\partial eB_{\geq X}$ is the face poset of a regular CW decomposition of the boundary sphere of this ball, whence the notation.   For example, $\Sigma_n$ is a CW left regular band; each $(\Sigma_n)_{\geq X}$ is the face poset of a permutohedron.  More generally, the face semigroup of any hyperplane arrangement or the monoid of covectors of an oriented matroid is a CW left regular band.  If $B$ is a CW left regular band, then $\Lambda(B)$ is a graded poset~\cite[Theorem~6.2]{ourmemoirs}.  Moreover, if $X\leq Y$, then $e_YB_{\geq X}$ is the face poset of a ball of dimension $\rk[X,Y]$.   Further examples and details can be found in~\cite{ourmemoirs}.

Let us say that a left regular band of groups $M$ is a \emph{CW left regular band of groups} if its idempotents $B$ form a CW left regular band.  For example, Hsiao's monoid $\Sigma_n^G$ is a CW left regular band of groups.  We shall compute all $\Ext$ between simple modules for a CW left regular band of groups and compute a quiver presentation for a CW left regular band of abelian groups over algebraically closed fields of good characteristic.  This will include $\Sigma_n^G$ for $G$ abelian.  We will also show these algebras are Koszul and recognize their Koszul duals as incidence algebras of posets.  In a later section, we shall provide explicit minimal projective resolutions of all the simple modules (in good characteristic) using the action of the monoid on its associated $CW$-complexes.

\begin{Thm}\label{t:first.stab}
Let $M$ be a CW left regular band of groups with CW left regular band of idempotents $B$ and $K$ a field.  Fix an idempotent $e_X$ with $Be_X=X$ for each $X\in \Lambda(B)$ and put $G_X = G_{e_X}$.   If $V,W$ are simple $KG_X$-, $KG_Y$-modules, respectively, inflated to $KM$-modules in the usual way,  then $\Ext^r_{KM}(V,W)=0$ unless $X\leq Y$ and $r\geq \rk[X,Y]$, in which case
\[\Ext^r_{KM}(V,W) \cong \Ext^{r-\rk[X,Y]}_{KG_Y}(V,W)\] where we note that $V$ is $KG_Y$-module via inflation along the homomorphism $\rho_X(g)=e_Xg$.  In particular, if the characteristic of $K$ does not divide $|G|$, then $\Ext^r_{K\Sigma_n^G}(V,W)=0$ unless $X\leq Y$ and $r=\rk[X,Y]$, in which case $\Ext^r_{K}(V,W)\cong \Hom_{KG_Y}(V,W)$.
\end{Thm}
\begin{proof}
Necessity of $X\leq Y$ and the case that $X=Y$ follows from Theorem~\ref{t:left.duo}.  Assume now that $X<Y$.  Then $\Delta(\partial e_YB_{\geq X})$ is a triangulation of a $(\rk[X,Y]-1)$-sphere by the axiomatization of a CW poset. Also $G_Y$ acts trivially on the corresponding simplicial complex by Proposition~\ref{p:lrbg.pi}.  Thus $\til H_{q-1}(\Delta(\partial e_YB_{\geq X})),K)=0$ unless $q-1=\rk[X,Y]-1$, that is $q=\rk[X,Y]$, in which case $\til H_{q-1}(\Delta(\partial e_YB_{\geq X}),K)\cong K$ is the trivial $KG_Y$-module.  Thus, the spectral sequence in Theorem~\ref{t:left.duo}(3) collapses on the $E_2$-page on the line $q=\rk[X,Y]$ to yield
\begin{align*}
\Ext_{KM}^{r}(V,W)&\cong \Ext^{r-\rk[X,Y]}_{KG_Y}(K, V^*\otimes_K W)= H^{r-\rk[X,Y]}(G_Y,V^\ast\otimes_K W)  \\ &\cong H^{r-\rk[X,Y]}(G_Y,\Hom_K(V,W))\cong \Ext^{r-\rk[X,Y]}_{KG_Y}(V,W)
\end{align*}
by~\cite[Corollary~4.5]{affinetoprep} applied to $G_Y$.  In particular, if the characteristic of $K$ does not divide $|G|$, these spaces vanish except when $r=\rk[X,Y]$, in which case $\Ext^r_{KM}(V,W)\cong \Hom_{KG_Y}(V,W)$. Alternatively, this last conclusion follows from Theorem~\ref{t:lrbg.ext} upon noting that since $\Delta(\partial e_YB_{\geq X})$ is a triangulation of a $(\rk[X,Y]-1)$-sphere, its only nonvanishing reduced Betti number is $\til\beta_{r-1}=1$.
\end{proof}

We can now describe exactly the quiver of a CW left regular band of abelian groups over an algebraically closed field whose characteristic does not divide the order of any subgroup.

\begin{Cor}\label{c:cwlrbg.quiver}
Let $M$ be a CW left regular band of abelian groups with CW left regular band of idempotents $B$ and $K$ an algebraically closed field whose characteristic divides the order of no maximal subgroup of $M$.  Then the quiver of $M$ is the Hasse diagram (with edges oriented upward) of the poset $P(M)$ consisting of pairs $(X,\chi)$ with $X\in \Lambda(B)$ and $\chi\in \wh G_X$ with the ordering $(X,\chi)\leq (Y,\theta)$ if $X\leq Y$ and $\theta=\chi\circ\rho_X$.
\end{Cor}
\begin{proof}
The poset $\Lambda(B)$ is graded by~\cite[Theorem~6.2]{ourmemoirs} and hence $P(M)$ is graded with $\rk[(X,\chi),(Y,\theta)]=\rk[X,Y]$ by Proposition~\ref{p:look.up}, its proof and the discussion thereafter.
It follows from Theorem~\ref{t:first.stab} that $\Ext^1_{KM}((X,\chi),(Y,\theta))$ is one-di\-men\-sio\-nal if $\rk[(X,\chi),(Y,\theta)]=1$ and is zero otherwise.  The corollary follows from this.
\end{proof}

For a CW left regular band, these results specialize to the results of~\cite{ourmemoirs}.
Let us rephrase our results for the special case of $\Sigma_n^G$.

\begin{Cor}\label{c:first.stab}
Let $G$ be a finite group and $K$ a field.  If $V,W$ are simple $K\Sigma_n^G$-modules with apexes $X,Y\in \Pi_n$ respectively, then $\Ext^r_{K\Sigma_n^G}(V,W)=0$ unless $X\leq Y$ and $r\geq |X|-|Y|$, in which case
\[\Ext^r_{K\Sigma_n^G}(V,W) \cong \Ext^{r-|X|+|Y|}_{KG_Y}(V,W)\] where we note that $W$ is a simple $KG_Y$-module, $V$ is a simple $KG_X$-module and $G_Y\leq G_X$.  In particular, if the characteristic of $K$ does not divide $|G|$, then $\Ext^r_{K\Sigma_n^G}(V,W)=0$ unless $X\leq Y$ and $r= |X|-|Y|$, in which case $\Ext^r_{K\Sigma_n^G}(V,W)\cong \Hom_{KG_Y}(V,W)$.
\end{Cor}
\begin{proof}
This is immediate from Theorem~\ref{t:first.stab} once we observe that $\rk[X,Y] = |X|-|Y|$.
\end{proof}

This, of course, specializes in the case that $G$ is trivial to the computation for $\Sigma_n$ in~\cite{Saliolahyperplane,MSS,ourmemoirs}

To give a more precise description of $\Ext$ between the simple $K\Sigma_n^G$-modules, it will be convenient to impose the condition that $K$ is algebraically closed so that each simple $KG^r$-module is an outer tensor product of simple $KG$-modules.  We shall need the following well-known fact;  a proof for the special case of endomorphisms can be found in~\cite[Corollary~17.7.7]{KapilovskyB}, but essentially the same proof works in the general case.  Recall that if $G, H$ are groups, then $K[G\times H]\cong KG\otimes_K KH$ as $K$-algebras.

\begin{Lemma}\label{l:direct.prod}
Let $G$ and $H$ be finite groups and $K$ a field.  If $V,W$ are finite dimensional $KG$-modules and $V',W'$ are finite dimensional $KH$-modules, then $\Hom_{K[G\times H]}(V\otimes_K V',W\otimes_K W')\cong \Hom_{KG}(V,W)\otimes_K \Hom_{KH}(V',W')$.
\end{Lemma}

 Recall that $K$ is a splitting field for a group $G$ if each simple $KG$-module is absolutely simple (i.e., remains simple after extending scalars to the algebraic closure). The following proposition is well known, cf.~\cite[Corollary~17.8.13]{KapilovskyB}.

\begin{Prop}\label{p:split.field}
Suppose that $G$ and $H$ are finite groups and $K$ is a splitting field for both $G$ and $H$.  Then $K$ is a splitting field for $G\times H$ and if $V_1,\ldots, V_n$ and $W_1,\ldots, W_m$ are complete sets of representatives of the isomorphism classes of simple $KG$-modules, respectively $KH$-modules, then $\{V_i\otimes_K W_j\mid 1\leq i\leq n, 1\leq j\leq m\}$ is a complete set of representatives of the isomorphism classes of simple $K[G\times H]$-modules.
\end{Prop}

Let us introduce some notation to set up our next theorem.  Let $G$ be a finite group  and  $K$ a splitting field for $G$.  Fix once and for all a set $A$ of representatives of the isomorphism classes of simple $KG$-modules.

If $X\in \Pi_n$ and $g\in G_X$, then for $P\in X$ we set $g_P$  to be the common value $g_p$ with $p\in P$ so that $G_X\cong G^{|X|}$ via $g\mapsto (g_P)_{P\in X}$.
The simple $K\Sigma_n^G$-modules (up to isomorphism) are indexed by pairs $(X,f)$ where $X\in \Pi_n$ and $f\colon X\to A$, that is, $f$ assigns a simple module to each partition block.  The corresponding simple $KG_X$-module is given by $\bigotimes_{P\in X} f(P)$ (which is a simple $KG^{|X|}$-module).  Concretely, if $P=\{P_1,\ldots, P_r\}$, then \[g(v_1\otimes\cdots\otimes v_r) = g_{P_1}v_1\otimes\cdots \otimes g_{P_r}v_r\] for $g\in G_X$ and $v_i\in f(P_i)$.  Of course, this inflates to a simple $K\Sigma_n^G$-module in the usual way.  For convenience, we shall adopt the notation $(X,f)$ for the corresponding simple module.

\begin{Thm}\label{t:Hsiao}
Let $G$ be a finite group and $K$ a splitting field for $G$ such that the characteristic of $K$ does not divide $|G|$. Retaining the previous notation, we have that if $(X,f)$ and $(Y,h)$ are simple $K\Sigma_n^G$-modules, then $\Ext^q_{K\Sigma_n^G}((X,f),(Y,h))=0$ unless $X\leq Y$ and $q = |X|-|Y|$, in which case
\[\Ext^q_{K\Sigma_n^G}((X,f),(Y,h)) \cong \bigotimes_{P\in Y} \Hom_{KG}\left(\bigotimes_{Q\in X,Q\subseteq P}f(Q),h(P)\right)\]
holds.
\end{Thm}
\begin{proof}
Corollary~\ref{c:first.stab} reduces the theorem to checking that
\[\Hom_{KG_Y}((X,f),(Y,h))\cong \bigotimes_{P\in Y} \Hom_{KG}\left(\bigotimes_{Q\in X,Q\subseteq P}f(Q),h(P)\right).\]  We claim that the restriction of $(X,f)$ to $G_Y\leq G_X$ is isomorphic to $\bigotimes_{P\in Y}\bigotimes_{Q\in X, Q\subseteq P}f(Q)$, where the inner tensor product $\bigotimes_{Q\in X, Q\subseteq P}f(Q)$ is a $KG$-module.  Assuming this claim, the theorem follows from Lemma~\ref{l:direct.prod}.   Indeed, if $g\in G_Y$ and $Q\in X$, then $g_Q = g_P$ for the unique $P\in Y$ with $Q\subseteq P$.  So if $Q=\{Q_1,\ldots, Q_r\}$ and if $Q_i\subseteq P_{s(i)}\in Y$, then we have, for $g\in G_Y$, that \[g(v_1\otimes\cdots\otimes v_r) = g_{P_{s(1)}}v_1\otimes \cdots \otimes g_{P_{s(r)}}v_r\]  for $v_i\in f(Q_i)$.  Grouping together the tensor factors over the fibers of $s$, we see that $(X,f)\cong  \bigotimes_{P\in Y}\bigotimes_{Q\in X, Q\subseteq P}f(Q)$ as a representation of $KG_Y$.  This completes the proof.
\end{proof}

The special case where $q=1$ is the content of~\cite[Theorem~6.3]{rrbg}, which was proved using different methods.

Theorem~\ref{t:Hsiao} greatly simplifies in the case that $G$ is abelian.  In this case, each simple $KG$-module is one-dimensional, and so the tensor product of simple $KG$-modules is again simple.  Let us identify the set of isomorphism classes of simple $KG$-modules with the dual group $\wh G=\Hom(G,K^\times)$ of (multiplicative) characters of $G$.  The tensor product of modules then corresponds to the pointwise product on $\wh G$. Simple $K\Sigma_n^G$-modules now correspond to pairs $(X,f)$ where $X\in \Pi_n$ and $f\colon X\to \wh G$.

\begin{Cor}\label{c:abelian}
Let $G$ be a finite abelian group and $K$ a splitting field for $G$ such that the characteristic of $K$ does not divide $|G|$. Retaining the previous notation, we have that if $(X,f)$ and $(Y,h)$ are simple $K\Sigma_n^G$-modules, then $\Ext^q_{K\Sigma_n^G}((X,f),(Y,h))=0$ unless $X\leq Y$, $q = |X|-|Y|$ and $h(P) = \prod_{Q\in X, Q\subseteq P}f(Q)$ in $\wh G$ for each $P\in Y$, in which case $\Ext^q_{K\Sigma_n^G}((X,f),(Y,h))\cong K$.
\end{Cor}
\begin{proof}
Note that $\bigotimes_{Q\in X, Q\subseteq P}f(Q)$ is the one-dimensional simple $KG$-module with character $\prod_{Q\in X, Q\subseteq P}f(Q)$.  The result now follows from Theorem~\ref{t:Hsiao} and Schur's lemma.
\end{proof}

Let $K$ be an algebraically closed field whose characteristic does not divide $|G|$ where $G$ is a finite abelian group.  Note that the poset $P(\Sigma_n^G)$ can be identified with pairs $(X,f)$ with $X\in \Pi_n$ and $f\colon X\to \wh G$.  One has $(X,f)\leq (Y,h)$ if and only if the partition $X$ refines $Y$ and $h(B) = \prod_{B'\in X, B'\subseteq B}f(B')$ for $B\in Y$.  The Hasse diagram of this partial order (with edges oriented upward) is the quiver of $K\Sigma_n^G$, which agrees with the description in~\cite{rrbg}.

\subsection{A quiver presentation and Koszul duality for the monoid of ordered $G$-partitions}
We begin this section by making explicit a result from~\cite{ourmemoirs}, which unfortunately is not proved in fully generality there.

If $P$ is a finite poset and $K$ is a field, the incidence algebra $I(P,K)$ of $P$ over $K$ is the algebra of all mappings $f\colon P\times P\to K$ such that $f(p,q)=0$ unless $p\leq q$ with product the convolution \[f\ast g(p,q) = \sum_{p\leq z\leq q}f(p,z)g(z,q).\]  This is a finite dimensional algebra with basis the functions $\delta_{p,q}$ with $p\leq q$ and identity $\sum_{p\in P}\delta_{p,p}$.  If $Q(P)$ is the Hasse diagram of $P$, with all edges oriented from smaller elements to bigger elements, then $I(P,K)\cong KQ(P)^{op}/I$ where $I$ is generated by all differences $p-q$ with $p,q$ coterminal directed paths.  Note that here we are following the usual convention for multiplication in incidence algebras, whereas in~\cite{ourmemoirs} we used the opposite convention on multiplication.  Thus our $I(P,K)$ coincides with what is written as $I(P^{op},K)$ in~\cite{ourmemoirs}.

Let $P$ be a (finite) graded poset and let $Q(P)$ be the Hasse diagram of $P$, made a quiver as above.  We shall write $q\leftarrow p$ for the edge from $p$ to $q$ if $q$ covers $p$.  Let $R$ be the system of relations
\begin{equation}\label{eq:quiver.rels}
r_{p,q}=\sum_{p<z<q}(q\leftarrow z\leftarrow p),\quad \rk[p,q]=2.
\end{equation}
Note that the ideal $I$ generated by $R$ is homogeneous and generated in degree $2$, and so $KQ(P)/I$ is a quadratic algebra.  Our goal is to axiomatize those split basic algebras isomorphic to $KQ(P)/I$ and to prove that if each open interval of $P$ is Cohen-Macaulay, then $KQ(P)/I$ is Koszul and is the Koszul dual of the incidence algebra of $P$.  The following is a revamped version of~\cite[Theorem~6.1]{ourmemoirs}.

\begin{Thm}\label{t:quiverpres}
Let $P$ be a graded poset with oriented Hasse diagram $Q(P)$.  Let $A$ be a split basic algebra over a field $K$ such that there is a bijection $p\mapsto S_p$ from $P$ to the set of isomorphism classes of simple $A$-modules such that:
\begin{enumerate}
\item  $\Ext_A^1(S_p,S_q) \cong \begin{cases} K, &\text{if}\ \rk[p,q]=1\\ 0, & \text{else.}\end{cases}$
\item $\Ext^2_A(S_p,S_q) \cong \begin{cases} K, & \text{if}\ \rk[p,q]=2\\ 0, & \text{else.}\end{cases}$
\item There is a decomposition $1=\sum_{p\in P}\eta_p$ into orthogonal primitive idempotents with $A\eta_p/\rad(A)\eta_p\cong S_p$ and an $A$-module homomorphism $\partial\colon A\to A$ such that:
\begin{itemize}
\item [(a)] $\partial^2=0$;
\item [(b)] $\partial(\eta_p)\eta_p=0$, for all $p\in P$;
\item [(c)] $\partial(\eta_q)\eta_p\neq 0$ if $\rk[p,q]=1$.
\end{itemize}
\end{enumerate}
Then $A$ has a quiver presentation $(Q(P),R)$ where $R$ is  the system of  relations \eqref{eq:quiver.rels}.  In particular, the ideal $I$ generated by $R$ is a homogeneous ideal  and $A\cong KQ(P)/I$ is a quadratic algebra.  Conversely, if $A=KQ/I$ (where $I$ is generated by $R$), then $A$ satisfies (1)--(3).
\end{Thm}
\begin{proof}
Suppose that $A$ satisfies (1)--(3).  Then we can identify $Q(A)$ with $Q(P)$ by (1).
Set $Q=Q(A)=Q(P)$ and
suppose that there is an edge $q\leftarrow p$ in $Q$.  Then since $P$ is graded there are no paths of length greater than one in $Q$ from $p$ to $q$, and so $\eta_q\rad^2(A)\eta_p=0$.  Therefore, we have by (1) that
\begin{align*}
1&=\dim_K\Ext^1_{A}(S_p,S_q)= \dim_K\eta_q[\rad(A)/\rad^2(A)]\eta_p \\ &= \dim_K\eta_qA\eta_p.
\end{align*}
Observe that, by (3)(c), $0\neq \partial(\eta_q)\eta_p = \eta_q\partial(\eta_q)\eta_p$, and so $\{\partial(\eta_q)\eta_p\}$ is a basis for $\eta_qA\eta_p\cong \eta_q[\rad(A)/\rad^2(A)]\eta_p$.  Thus we have a surjective homomorphism $\p\colon KQ\to A$ given by $\p(\varepsilon_p)=\eta_p$ and $\p(q\leftarrow p) = \partial(\eta_q)\eta_p$ with $\ker \p$ admissible by the general theory of split basic finite dimensional algebras~\cite{assem,benson}.

Next we show that $r_{p,q}\in \ker \p$ if $\rk[p,q]=2$. Note that
\begin{equation}\label{imageofrelation}
\p(r_{p,q}) = \sum_{p< z< q} \partial(\eta_q)\eta_z\partial(\eta_z)\eta_p =\sum_{p< z< q} \eta_q\partial(\eta_q)\eta_z\partial(\eta_z)\eta_p.
\end{equation}
Notice that if $z\in P$ and $p<z<q$ is false, then \[\eta_q\partial(\eta_q)\eta_z\partial(\eta_z)\eta_p=0.\]  Indeed, if $z\notin [p,q]$, then there is no path in $Q$ from $p$ to $q$ through $z$, and so $\eta_qA\eta_zA\eta_p=0$.  If $z\in \{p,q\}$ we use that $\partial(\eta_p)\eta_p=0=\partial(\eta_q)\eta_q$ by (3)(b). Thus the right hand side of \eqref{imageofrelation} is equal to
\begin{align*}
\sum_{z\in P} \eta_q\partial(\eta_q)\eta_z\partial(\eta_z)\eta_p &= \sum_{z\in P}\partial[\partial(\eta_q)\eta_z]\eta_p\\ &= \partial\left[\partial(\eta_q)\cdot\sum_{z\in P}\eta_z\right]\eta_p\\ &=\partial^2(\eta_q)\eta_p \\ &=0
\end{align*}
using that $\partial^2=0$ and $\sum_{z\in P}\eta_z=1$.
This proves that $\p(r_{p,q})=0$.

To complete the proof of the forward implication, let $R'$ be a minimal system of relations for $I=\ker \p$. We claim that \[R'=\{c_{p,q}\cdot r_{p,q}\mid \rk[p,q]=2\}\] where each $c_{p,q}\in K\setminus \{0\}$.  It will then follow that the $r_{p,q}$ also form a minimal system of relations for $I$.

By Theorem~\ref{t:bongartz} and (2) we have that $R' = \{s_{p,q}\mid \rk[p,q]=2\}$ where $s_{p,q}$ is a non-zero linear combination of paths from $p$ to $q$ in $Q$, necessarily of length $2$ since $P$ is graded. It follows that if $\rk[p,q]=2$, then
\begin{equation*}
r_{p,q} = \sum_{\rk[u,v]=2}p_{u,v}s_{u,v}q_{u,v}
\end{equation*}
where $p_{u,v}\in \varepsilon_qKQ\varepsilon_v$ and $q_{u,v}\in \varepsilon_uKQ\varepsilon_p$.  But since each path in $r_{p,q}$ has length $2$ and each path in $s_{u,v}$ has length $2$, we conclude that $p_{u,v}s_{u,v}q_{u,v}=0$ if $(u,v)\neq (p,q)$ and that $r_{p,q}=k_{p,q}\cdot s_{p,q}$ for some $k_{p,q}\in K\setminus 0$.  This establishes the claim and completes the proof of the forward implication.

Suppose conversely that $A=KQ/I$ where $I$ is the ideal generated by the system of relations \eqref{eq:quiver.rels}.  We will abuse notation and identify vertices and arrows of $Q$ with their images in $KQ/I$.   Put $S_p= (KQ/I)\varepsilon_p$. Then since the quiver of $A$ has an edge from $p$ to $q$ if and only if $\rk[p,q]=1$, we deduce that (1) holds.  The argument of the preceding paragraph with $R'$ taken to be a subset of $R$ shows that $R$ is a minimal set of relations, and so (2) holds by Theorem~\ref{t:bongartz}.  Finally, we take $\eta_p=\varepsilon_p+I$ and we define $\partial$ on $KQ$ on the $KQ$-basis $1$ by $\partial 1 = \sum_{\rk[p,q]=1}(q\leftarrow p)$.  In general, $\partial x = x\partial 1$, and so it is immediate that $\partial I\subseteq I$.  Hence we have an induced module homomorphism $\partial\colon A\to A$.  To show that $\partial^2=0$, observe that
\begin{align*}
\partial^2(1) &= \partial\sum_{\rk[p,q]=1}(q\leftarrow p) = \sum_{\rk[p,q]=1}(q\leftarrow p)\cdot \sum_{\rk[p',q']=1}(q'\leftarrow p') \\ &= \sum_{\rk[p',q]=2}\sum_{p'<z<q} (q\leftarrow z\leftarrow p')=\sum_{\rk[p',q]=2}r_{p',q}=0.
\end{align*}
Next observe that $\partial(\eta_p)\eta_p = \eta_p\sum_{\rk[u,v]=1}(v\leftarrow u)\eta_p =0$, as $p=v=u=p$ is impossible if $\rk[u,v]=1$.  Also, if $\rk[p,q]=1$, then $\partial(\eta_q)\eta_p = \eta_q\sum_{\rk[u,v]=1}(v\leftarrow u)\eta_p = q\leftarrow p\neq 0$.  Thus $KQ/I$ satisfies (1)--(3) (and, moreover, the isomorphism constructed in the previous part based on this definition of $\partial$ will be the identity map).
\end{proof}

A maximal face of a simplicial complex is called a \emph{facet}.
A simplicial complex $K$ is called a \emph{chamber complex} if it is pure (all facets have the same dimension) and any two facets can be connected by a gallery.  Recall that if $C,C'$ are facets, then a \emph{gallery} from $C$ to $C'$ is a sequence $C=C_0,\ldots,C_n=C'$ of facets such that $C_i,C_{i+1}$ are distinct and adjacent, where  $C_i,C_{i+1}$ are \emph{adjacent} if they have a common codimension one face. Let $P$ be a graded poset. Then $P$ is said to be \emph{strongly connected} if $\Delta([p,q])$ is a chamber complex for each closed interval $[p,q]$ of $P$.  Note that if $\Delta((p,q))$ is a chamber complex, then so is $\Delta([p,q])$ as is easily checked (or see~\cite[Remark~7.3]{ourmemoirs}).

The reader is referred to~\cite{topmethods} for the definition of a Cohen-Macaulay poset over a field $K$.  What we shall need is that the order complex of a Cohen-Macaulay poset is a chamber complex~\cite[Proposition~11.7]{topmethods} (note that a poset whose order complex is a chamber complex is called strongly connected in~\cite{topmethods}).  We shall also need that if $P$ is a poset, then $I(P,K)$ is Koszul if and only if $(p,q)$ is Cohen-Macaulay over $K$ for all $p<q$ in $P$~\cite{polo,woodcock}.  By the above discussion, if $(p,q)$ is Cohen-Macaulay for all open intervals in $P$, then $\Delta(P)$ is strongly connected in our sense.  It then follows that if $Q$ is the Hasse diagram of $P$ with the edges oriented from smaller elements to larger elements, then the Koszul dual of $I(P,K)$ is $KQ/I$ where $I$ is generated by the system of relations \eqref{eq:quiver.rels} by~\cite[Theorem~7.5]{ourmemoirs} (where we note that the quadratic dual coincides with the Koszul dual for a Koszul algebra). In light of Theorem~\ref{t:quiverpres}, we may summarize this as follows.

\begin{Thm}\label{t:quiverpres.koszul}
Let $P$ be a graded poset whose open intervals are Cohen-Macaulay over $K$.  Let $Q(P)$ be the oriented Hasse diagram of $P$ (where edges go up in the order).  Let $A$ be a split basic algebra over a field $K$ such that there is a bijection $p\mapsto S_p$ from $P$ to the set of isomorphism classes of simple $A$-modules such that:
\begin{enumerate}
\item  $\Ext_A^1(S_p,S_q) \cong \begin{cases} K, &\text{if}, \rk[p,q]=1\\ 0, & \text{else.}\end{cases}$
\item $\Ext^2_A(S_p,S_q) \cong \begin{cases} K, & \text{if}\ \rk[p,q]=2\\ 0, & \text{else.}\end{cases}$
\item There is a decomposition $1=\sum_{p\in P}\eta_p$ into orthogonal primitive idempotents with $A\eta_p/\rad(A)\eta_p\cong S_p$ and an $A$-module homomorphism $\partial\colon A\to A$ such that:
\begin{itemize}
\item [(a)] $\partial^2=0$;
\item [(b)] $\partial(\eta_p)\eta_p=0$, for all $p\in P$;
\item [(c)] $\partial(\eta_q)\eta_p\neq 0$ if $\rk[p,q]=1$.
\end{itemize}
\end{enumerate}
Then $A$ has a quiver presentation $(Q(P),R)$ where $R$ is  the system of  relations \eqref{eq:quiver.rels}.  Moreover,  $A$ is a Koszul algebra with Koszul dual the incidence algebra $I(P,K)$.
\end{Thm}

We shall apply Theorem~\ref{t:quiverpres.koszul} to CW left regular bands of abelian groups in good characteristic to obtain a quiver presentation, generalizing~\cite[Theorems~6.2 and~7.12]{ourmemoirs} from the case of CW left regular bands.

Before doing so, we shall need some purely algebraic properties of the algebra of a left regular band of groups.

\begin{Prop}\label{p:move.past}
Let $M$ be a left regular band of groups with left regular band of idempotents $B$ and $K$ a commutative ring.  Let $e\in B$, $a\in KG_e$ and $x\in KB$.  Then $ax = exa$.
\end{Prop}
\begin{proof}
If $g\in G_e$ and $b\in B$, then $gbg^{\dagger}=eb$ by Proposition~\ref{p:lrbg.pi}, and so $ebg=gbg^\dagger g=gbg^{\omega}=gb$ (where the last equality was observed in the beginning of the proof of Proposition~\ref{p:identify.poset}). It then follows easily by induction on the size of the support that $gx=exg$ for all $x\in KB$, and then again by induction on the size of the support that $ax=exa$ for all $a\in KG_e$.
\end{proof}

\begin{Thm}\label{t:quiver.pres.CW.lrbg}
Let $M$ be a CW left regular band of abelian groups with CW left regular band of idempotents $B$ and $K$ an algebraically closed field whose characteristic divides the order of no maximal subgroup of $G$.  Let $P(M)$ be the poset of pairs $(X,\chi)$ with $\chi\in \wh G_X$ ordered by $(X,\chi)\leq (Y,\theta)$ if $X\leq Y$ and $\theta=\chi\circ \rho_X$.
\begin{enumerate}
  \item $P(M)$ is graded and each open interval in $P(M)$ is Cohen-Macaulay.
  \item $KM$ is a Koszul algebra and is the Koszul dual of the incidence algebra $I(P(M),K)$.
  \item $KM$ has quiver presentation $(Q(P(M)),R)$ where $Q(P(M))$ is the Hasse diagram of $P(M)$ with edges oriented upward and $R$ is the minimal system of relations given by \eqref{eq:quiver.rels}.
\end{enumerate}
\end{Thm}
\begin{proof}
The poset $\Lambda(B)$ is graded by~\cite[Theorem~6.2]{ourmemoirs}, and each of its open intervals is Cohen-Macaulay by~\cite[Theorem~7.14]{ourmemoirs}. Hence the same is true for $P(M)$ by Proposition~\ref{p:look.up} and the discussion thereafter.  In light of Theorem~\ref{t:first.stab} and Corollary~\ref{c:cwlrbg.quiver}, the first two conditions of Theorem~\ref{t:quiverpres.koszul} are satisfied.  Let $\{\gamma_{X,\chi}\mid (X,\chi)\in P(M)\}$ be the complete set of orthogonal primitive idempotents constructed in Theorem~\ref{t:complete.set.left.duo.reg}. Note that $\gamma_{X,\chi}=\eta_X\theta_{X,\chi}\eta_X$ where $\{\eta_X\mid X\in \Lambda(B)\}$ are as per \eqref{eq:defidempotents}, and $\theta_{X,\chi}$ is the primitive idempotent of $KG_X$ corresponding to $\chi$.  We remark that if $a\in KG_X$, then $a\theta_{X,\chi} = \chi(a)\theta_{X,\chi}$  since $\chi$ is the character of a one-dimensional simple module.  By Proposition~\ref{p:move.past} (since $\eta_X\in kB$), we have that $\gamma_{X,\chi}=\eta_X\theta_{X,\chi}\eta_X= \eta_Xe_X\eta_X\theta_{X,\chi} = \eta_X\theta_{X,\chi}$, and we shall also use this simplified formula.

The proof of~\cite[Theorem~6.2]{ourmemoirs} constructs a differential $\partial\colon KB\to KB$ with $\partial^2=0$, $\partial(\eta_X)\eta_X=0$ and $\partial(\eta_Y)\eta_X\neq 0$ when $\rk[X,Y]=1$.  We can extend $\partial$ to a differential $\partial \colon KM\to KM$ using that $KM$ is a free $KM$-module with basis $1$ and we already have $\partial(1)$ defined with the property $\partial^2(1)=0$.  It remains to check that (3)(b)--(c) of Theorem~\ref{t:quiverpres.koszul} hold (with $\eta_{X,\chi} = \gamma_{X,\chi}$).    First we compute $\partial(\gamma_{X,\chi})\gamma_{X,\chi} = \eta_X\theta_{X,\chi}\partial(\eta_X)\eta_X\theta_{X,\chi}=0$.    Suppose now that $\rk[(X,\chi),(Y,\psi)]=1$.  Then $\rk[X,Y]=1$ and $\psi=\chi\circ \rho_X$.  Therefore, we have $e_X\theta_{Y,\psi}= \rho_X(\theta_{Y,\psi}) = \theta_{X,\chi}$ as $\chi$ inflates to $\psi$ under $\rho_X$.  It follows  using Proposition~\ref{p:move.past} (as $\partial(\eta_Y)\eta_X\in kB$) that
\begin{align*}
\partial(\gamma_{Y,\psi})\gamma_{X,\chi} &= \eta_Y\theta_{Y,\psi}\partial(\eta_Y)\eta_X\theta_{X,\chi}
= \eta_Ye_Y\partial(\eta_Y)\eta_X\theta_{Y,\psi}\theta_{X,\chi} \\ &=\partial(\eta_Y)\eta_Xe_X\theta_{Y,\psi}\theta_{X,\chi}= \partial(\eta_Y)\eta_X\theta_{X,\chi}.
\end{align*}
Now $B\cap L_{e_X}$ is a basis for $KL_{e_X}$  as a right $KG_{e_X}$-module since the $\mathscr R$-classes are the $G_{e_X}$-orbits. So $KL_{e_X}\otimes_{KG_X}KG_X\theta_{X,\chi}$ has $K$-basis the elementary tensors $b\otimes \theta_{X,\chi}$ with $b\in B\cap L_{e_X}$.  The isomorphism $\Ind_{e_X}(\chi)\to KM\eta_X\theta_{X,\chi}$ sends $b\otimes \theta_{X,\chi}$ to $b\eta_X\theta_{X,\chi}$ by the last paragraph of the proof of Theorem~\ref{t:complete.set.left.duo.reg}.  Since $\partial(\eta_Y)\eta_X$ is a nonzero linear combination of elements of the form $b\eta_X$ with $b\in B\cap L_{e_X}$, it follows that $\partial(\eta_Y)\eta_X\theta_{X,\chi}$ is a nonzero linear combination of elements of elements of the form $b\eta_X\theta_{X,\chi}$ with $b\in B\cap L_{e_X}$ and hence is nonzero. Therefore, $\partial(\gamma_{Y,\psi})\gamma_{X,\chi}\neq 0$,  completing the proof that (3) of Theorem~\ref{t:quiverpres.koszul} holds.  The theorem now follows from Theorem~\ref{t:quiverpres.koszul}.
 \end{proof}

We now spell out what the theorem says for Hsiao's monoid of ordered $G$-partitions, when $G$ is abelian, and also describe the connected components of the quiver (that is, the blocks).

\begin{Thm}\label{t:Hsiao.final}
Let $G$ be a finite abelian group and $K$ an algebraically closed field whose characteristic does not divide the order of $G$.  Then the quiver of $K\Sigma_n^G$ is the Hasse diagram $Q(P)$ (with edges oriented upward) of the poset $P$ consisting of pairs $(X,f)$ with $X$ a set partition of $\{1,\ldots, n\}$ and $f\colon X\to \wh G$ a mapping, where $\wh G=\Hom(G,K^\times)$.  One has $(X,f)\leq (Y,g)$ if and only if $X$ is a finer partition than $Y$, and  $g(B)=\prod_{B'\in X,B'\subseteq B}f(B')$ for all $B\in Y$.  Furthermore, $K\Sigma_n^G$ is isomorphic $KQ(P)/I$ where $I$ is generated by the relations that sum all paths between $(X,f),(Y,g)$ when $(X,f)\leq (Y,g)$ and $|X|-|Y|=2$.  Moreover, $K\Sigma_n^G$ is a Koszul algebra with Koszul dual the incidence algebra of the poset $P$.   The connected components of $Q(P)$ are in bijection with $\wh G$ (and hence with $G$).
\end{Thm}
\begin{proof}
Everything except the final statement is immediate from Theorem~\ref{t:quiver.pres.CW.lrbg} and the discussion following Corollary~\ref{c:abelian}.  Define a map $\psi\colon P\to \wh G$ by $\psi(X,f) = \prod_{B\in X}f(B)$.  Then $\psi$ is constant along edges of $Q(P)$ and hence on connected components.  Moreover, $(X,f)$ has a directed path to $(\{\{1,\ldots, n\}\},\psi(X,f))$.  It follows that $\psi$ induces a bijection between path components of $Q(P)$ and $\wh G$, as required.
\end{proof}

\section{Minimal projective resolutions from topology}
We now give finite projective resolutions, using topology, of all the simple modules for finite regular left duo monoids over fields of good characteristic, generalizing~\cite[Theorem~5.12]{ourmemoirs} from the left regular band case. Then, we construct using regular cell complexes the minimal projective resolutions of the simple modules for algebras of CW left regular bands of abelian groups over algebraically closed fields of good characteristic, generalizing~\cite[Corollary~5.32]{ourmemoirs} for the case of CW left regular bands.

Let $M$ be a finite regular left duo monoid with left regular band of idempotents $B$ and let $X\in \Lambda(B)$.  Then we have the contraction $M_{\geq X} = \{m\in M\mid \sigma(m)\geq X\}$, which is also a regular left duo monoid with band of idempotents $B_{\geq X}$ and minimal ideal $\sigma\inv(X)$.  If $K$ is a field, then we have a homomorphism $\rho_X\colon KM\to KM_{\geq X}$ given by
\[\rho_X(m) = \begin{cases} m, & \text{if}\ \sigma(m)\geq X\\ 0, & \text{else.}\end{cases}\]  There is also the homomorphism $\psi_X\colon KM_{\geq X}\to KG(M_{\geq X})$ coming from the group completion and every simple $KM$-module with apex $X$ is inflated via $\psi_X\rho_X$.   Theorem~\cite[Theorem~6.4]{affinetoprep} will then apply to provide projective resolutions of these simples as $KM_{\geq X}$-modules, and  so our first goal is to show that $KM_{\geq X}$ is a projective $KM$-module, which will then allow us to inflate projective $KM_{\geq X}$-module resolutions to $KM$.

Recall that a \emph{principal series} for a monoid $M$ is an unrefinable series of two-sided ideals $\emptyset= I_0\subsetneq I_1\subsetneq\cdots\subsetneq I_n$.  Every finite monoid has a principal series and for any principal series, the sets $I_j\setminus I_{j-1}$ are the $\mathscr J$-classes of $M$ (without repetition), cf.~\cite[Proposition~1.20]{repbook}.

\begin{Prop}\label{p:contract.proj}
Let $M$ be a  finite regular left duo monoid with left regular band of idempotents $B$ and $K$ a field.  Then $KM_{\geq X}$ is a projective $KM$-module for any $X\in \Lambda(B)$.
\end{Prop}
\begin{proof}
First note that if $e\in E(M)$, then $KL_e$ is projective by Proposition~\ref{p:schutz.proj}.  Next observe that if $\emptyset= I_0\subsetneq I_1\subsetneq\cdots\subsetneq I_n=M_{\geq X}$ is a principal series, then each $I_k\setminus I_{k-1}=J_{e_k}=L_{e_k}$ for some idempotent $e_k\in M_{\geq X}$ because $M$ is regular and left duo.  Thus $KI_k/K_{k-1}\cong KL_{e_k}$ as a $KM$-module (since $\sigma_{M_{\geq X}} =\sigma|_{M_{\geq X}}$). Since a module filtered by projectives is projective (cf.~\cite[Lemma~5.9]{ourmemoirs}),  $KM_{\geq X}$ is projective. 	
\end{proof}

The following theorem generalizes~\cite[Theorem~5.12]{ourmemoirs} beyond the case of left regular bands.

\begin{Thm}\label{t:proj.res.duo}
Let $M$ be a finite regular left duo monoid with left regular band of idempotents $B$ and $K$ a field whose characteristic does not divide the order of any maximal subgroup of $M$.   Let $X\in \Lambda(B)$ and $V$  simple $KG_X$-module, where $G_X$ is a maximal subgroup of the $\eL$-class $\sigma\inv(X)$.
Then $C_\bullet(\Delta(B_{\geq X}),V)\to V$ is a projective resolution of $V$, where $M_{\geq X}$ acts simplicially on $\Delta(B_{\geq X})$ as per Proposition~\ref{p:identify.poset}, and we view $C_q(\Delta(B_{\geq X}),V)$ as the tensor product $C_q(\Delta(B_{\geq X}),K)\otimes_K V$ as a $KM_{\geq X}$-module, which we then inflate to a $KM$-module via $\rho_X$.
\end{Thm}
\begin{proof}
It follows from Proposition~\ref{p:contract.proj} that any projective $KM_{\geq X}$-module inflates to a projective $KM$-module.  The result then follows from Proposition~\ref{p:identify.poset} and~\cite[Theorem~6.4]{affinetoprep}, since each induced module is projective by Corollary~\ref{c:proj.cover} (which shows that the induction of any simple, and hence semisimple, module is projective under our assumption on $K$).
\end{proof}

We end this section with a construction of the minimal projective resolutions of the simple modules for a CW left regular band of abelian groups, generalizing the case of CW left regular bands from~\cite{ourmemoirs}.

 Let $A$ be a finite dimensional $K$-algebra and let $M$ be a finite dimensional $A$-module.  A projective resolution $P_\bullet\to M$ is said to be \emph{minimal} if, for each $q>0$, one has $d(P_q)\subseteq \rad(P_{q-1})$ or, equivalently, $d\colon P_{q-1}\to d(P_{q-1})$ is a projective cover. Every finite dimensional module admits a minimal projective resolution, which is unique up to isomorphism.  The following proposition is stated in the context of group algebras in~\cite[Proposition~3.2.3]{cohomologyringbook}, but the proof given there is valid for any finite dimensional algebra.

\begin{Prop}\label{p:minresolution}
Let $A$ be a finite dimensional $K$-algebra over a field $K$, $M$ a finite dimensional $A$-module and $P_\bullet\to M$ a projective resolution.  Then the following are equivalent.
\begin{enumerate}
\item $P_\bullet\to M$ is the minimal resolution of $M$.
\item $\Hom_A(P_q,S)\cong\Ext_A^q(M,S)$ for any $q\geq 0$ and simple $A$-module $S$.
\item The coboundary map $d^*\colon \Hom_A(P_q,S)\to \Hom_A(P_{q+1},S)$ is the zero map for all simple $A$-modules $S$ and $q\geq 0$.
\item If $Q_\bullet\to M$ is a projective resolution of $M$, then the chain map $Q_\bullet\to P_\bullet$ lifting the identity map on $M$ is surjective in each degree $q\geq 0$.
\item If $Q_\bullet\to M$ is a projective resolution of $M$, then the chain map $P_\bullet\to Q_\bullet$ lifting the identity map on $M$ is injective in each degree $q\geq 0$.
\end{enumerate}
\end{Prop}

 If $B$ is a CW left regular band, then there is a regular CW complex $\Sigma(B)$ with face poset $B$.  Moreover, it is shown in~\cite{ourmemoirs} that there is an action of $B$ on $\Sigma(B)$ by regular cellular mappings (mappings that send open/closed $n$-cells onto open/closed $m$-cells with $m\leq n$) which induces the original action of $B$ on the face poset.  It is proved in~\cite[Corollary~5.32]{ourmemoirs} that the augmented cellular chain complex $C_{\bullet}(\Sigma(B),K)\to K$ is the minimal projective resolution of $K$.   More generally, it is shown that each projective  $KB_{\geq X}$-module inflates along the projection $KB\to KB_{\geq X}$ to a projective $KB$-module and the augmented cellular chain complex $C_{\bullet}(\Sigma(B_{\geq X}),K)\to K$ is the minimal projective resolution of the one-dimensional simple $KB$-module with apex $X$.  We now shall extend this to CW left regular bands of abelian groups over algebraically closed fields of good characteristic.  One could use these resolutions to prove directly that the algebra of a CW left regular band of abelian groups is Koszul, as was done for CW left regular bands in~\cite[Theorem~7.12]{ourmemoirs}, but we shall not do so here.

 \begin{Prop}\label{p:map.to.band}
 Let $M$ be a left regular band of groups with left regular band of idempotents $B$.  Then $\psi\colon M\to B$ given by $\psi(m)=m^{\omega}$ is a surjective  homomorphism fixing $B$ such that $mem^\dagger =\psi(m)e$ for all $m\in M$ and $e\in B$.  If $K$ is a field whose characteristic does not divide the order of any maximal subgroup of $M$, then $KB$ inflates to a projective $KM$-module under $\psi$.
 \end{Prop}
 \begin{proof}
Trivially, $\psi$ fixes $B$.  It is a monoid homomorphism by  Proposition~\ref{p:lrbg.pi}.  Moreover, $mem^\dagger =m^{\omega}e=\psi(m)e$ by Proposition~\ref{p:lrbg.pi}.

Assume now that the characteristic of $K$ does not divide the order of any maximal subgroup of $M$.
Let $L_{e_X}$ denote the $\mathscr L$-class of $e_X$ in $M$ and $L'_{e_X}$ denote the $\mathscr L$-class of $e_X$ in $B$ for $X\in \Lambda(B)$ and $e_X\in B$ with $X=Be_X$.  Then the module $KB$ is filtered by the modules $KL'_{e_X}$, and so it suffices to show that $KL'_{e_X}$ is a projective $KM$-module.  But notice that $\psi$ takes the right $G_X$-orbits of $L_{e_X}$ bijectively to $L'_{e_X}$ (since each $G_X$-orbit contains a unique idempotent),  and hence we have an isomorphism $\Ind_{e_X}(K) = KL_{e_X}\otimes_{KG_{e_X}}K\to KL'_{e_X}$.  But $\Ind_{e_X}(K)$ is projective by Corollary~\ref{c:proj.cover} since $K$ is a projective $KG_{e_X}$-module.
\end{proof}

Note that Proposition~\ref{p:map.to.band} implies that the action of $M$ on $B$ from Proposition~\ref{p:identify.poset} is inflated from the left action of $B$ on itself via $\psi$.

We now state our next main result, which applies, in particular, to Hsiao's monoid of ordered $G$-partitions when $G$ is abelian.

\begin{Thm}\label{t:min.resolutions}
Let $M$ be a CW left regular band of abelian groups with CW left regular band of idempotents $B$ and $K$ an algebraically closed field whose characteristic does not divide the order of any maximal subgroup.  Fix $e_X\in B$, for each $X\in \Lambda(B)$, with $Be_X=X$ and let $G_X=G_{e_X}$.  Let $V$ be a simple $KG_X$-module (which we then inflate to a simple $KM$-module).  Then $C_{\bullet}(\Sigma(B_{\geq X}),V)\to V$ is the minimal projective resolution of $V$ over $KM$, where we view $C_q(\Sigma(B_{\geq X}), K)$ as a $KM$-module via inflation along the projection $\psi\colon M\to B$ and $C_q(\Sigma(B_{\geq X}),V)=C_q(\Sigma(B_{\geq X}),K)\otimes_K V$ as a tensor product of $KM$-modules.
\end{Thm}
\begin{proof}
By Proposition~\ref{p:contract.proj}, we may assume without loss of generality that $X$ is the minimum element of $\Lambda(B)$, as a minimal projective resolution of $V$ over $KM_{\geq X}$ (which is a CW left regular band of abelian groups) then inflates to one over $KM$.
First note that since $KB$ is a projective $KM$-module via inflation along $\psi$ by Proposition~\ref{p:map.to.band}, we deduce from~\cite[Corollary~5.32]{ourmemoirs} that $C_{\bullet}(\Sigma(B),K)\to K$ is the minimal projective resolution of the trivial $KM$-module.  Then $C_{\bullet}(\Sigma(B),V)\to V$ (with the above described module structure) is a projective resolution of $V$  since tensoring with $V$ is exact and preserves projectivity by~\cite[Theorem~4.4]{affinetoprep}, as $V$ is inflated from a $KG(M)$-module by Proposition~\ref{p:completeness}.    We now check (2) of Proposition~\ref{p:minresolution}.  Let $W$ be a simple $KM$-module.  Note that $\dim_K V=1=\dim_K W$.   We have that $\Ext^q_{KM}(V,W) \cong \Ext^q_{KM}(K,\Hom_K(V,W))$ by~\cite[Corollary~4.5]{affinetoprep}.  Observe that $\Hom_K(V,W)$ is one-dimensional, and hence simple.  Thus $\Ext^q_{KM}(K,\Hom_K(V,W))\cong \Hom_{KM}(C_q(\Sigma(B),K),\Hom_K(V,W))$ by (2) of Proposition~\ref{p:minresolution} applied to the minimal resolution $C_{\bullet}(\Sigma(B),K)\to K$.  But $\Hom_{KM}(C_q(\Sigma(B),K),\Hom_K(V,W))\cong \Hom_{KM}(C_q(\Sigma(B),K)\otimes_K V,W)$ by~\cite[Theorem~4.4]{affinetoprep}.  Thus $\Ext^q_{KM}(V,W)\cong \Hom_{KM}(C_q(\Sigma(B),V),W)$, as required. This completes the proof.
\end{proof}


\begin{thebibliography}{10}

\bibitem{assem}
I.~Assem, D.~Simson, and A.~Skowro{\'n}ski.
\newblock {\em Elements of the representation theory of associative algebras.
  {V}ol. 1}, volume~65 of {\em London Mathematical Society Student Texts}.
\newblock Cambridge University Press, Cambridge, 2006.
\newblock Techniques of representation theory.

\bibitem{2023arXiv230614985B}
J.~{Bastidas}, S.~{Brauner}, and F.~{Saliola}.
\newblock {Left Regular Bands of Groups and the Mantaci--Reutenauer algebra}.
\newblock {\em arXiv e-prints}, June 2023.
\newblock arXiv:2306.14985.

\bibitem{BGSKoszul}
A.~Beilinson, V.~Ginzburg, and W.~Soergel.
\newblock Koszul duality patterns in representation theory.
\newblock {\em J. Amer. Math. Soc.}, 9(2):473--527, 1996.

\bibitem{benson}
D.~J. Benson.
\newblock {\em Representations and cohomology. {I}}, volume~30 of {\em
  Cambridge Studies in Advanced Mathematics}.
\newblock Cambridge University Press, Cambridge, second edition, 1998.
\newblock Basic representation theory of finite groups and associative
  algebras.

\bibitem{BHR}
P.~Bidigare, P.~Hanlon, and D.~Rockmore.
\newblock A combinatorial description of the spectrum for the {T}setlin library
  and its generalization to hyperplane arrangements.
\newblock {\em Duke Math. J.}, 99(1):135--174, 1999.

\bibitem{BidigareThesis}
T.~P. Bidigare.
\newblock {\em Hyperplane arrangement face algebras and their associated
  {M}arkov chains}.
\newblock ProQuest LLC, Ann Arbor, MI, 1997.
\newblock Thesis (Ph.D.)--University of Michigan.

\bibitem{BjornerCW}
A.~Bj{\"o}rner.
\newblock Posets, regular {CW} complexes and {B}ruhat order.
\newblock {\em European J. Combin.}, 5(1):7--16, 1984.

\bibitem{topmethods}
A.~Bj\"{o}rner.
\newblock Topological methods.
\newblock In {\em Handbook of combinatorics, {V}ol. 1, 2}, pages 1819--1872.
  Elsevier Sci. B. V., Amsterdam, 1995.

\bibitem{bjorner2}
A.~Bj{\"o}rner.
\newblock Random walks, arrangements, cell complexes, greedoids, and
  self-organizing libraries.
\newblock In {\em Building bridges}, volume~19 of {\em Bolyai Soc. Math.
  Stud.}, pages 165--203. Springer, Berlin, 2008.

\bibitem{Bongartz}
K.~Bongartz.
\newblock Algebras and quadratic forms.
\newblock {\em J. London Math. Soc. (2)}, 28(3):461--469, 1983.

\bibitem{Brown1}
K.~S. Brown.
\newblock Semigroups, rings, and {M}arkov chains.
\newblock {\em J. Theoret. Probab.}, 13(3):871--938, 2000.

\bibitem{Brown2}
K.~S. Brown.
\newblock Semigroup and ring theoretical methods in probability.
\newblock In {\em Representations of finite dimensional algebras and related
  topics in Lie theory and geometry}, volume~40 of {\em Fields Inst. Commun.},
  pages 3--26. Amer. Math. Soc., Providence, RI, 2004.

\bibitem{DiaconisBrown1}
K.~S. Brown and P.~Diaconis.
\newblock Random walks and hyperplane arrangements.
\newblock {\em Ann. Probab.}, 26(4):1813--1854, 1998.

\bibitem{cohomologyringbook}
J.~F. Carlson, L.~Townsley, L.~Valeri-Elizondo, and M.~Zhang.
\newblock {\em Cohomology rings of finite groups}, volume~3 of {\em Algebras
  and Applications}.
\newblock Kluwer Academic Publishers, Dordrecht, 2003.
\newblock With an appendix: Calculations of cohomology rings of groups of order
  dividing 64 by Carlson, Valeri-Elizondo and Zhang.

\bibitem{Graham}
R.~L. Graham.
\newblock On finite {$0$}-simple semigroups and graph theory.
\newblock {\em Math. Systems Theory}, 2:325--339, 1968.

\bibitem{TopFinite1}
R.~D. Gray and B.~Steinberg.
\newblock Topological finiteness properties of monoids, {I}: {F}oundations.
\newblock {\em Algebr. Geom. Topol.}, 22(7):3083--3170, 2022.

\bibitem{Green}
J.~A. Green.
\newblock On the structure of semigroups.
\newblock {\em Ann. of Math. (2)}, 54:163--172, 1951.

\bibitem{Hsiao}
S.~K. Hsiao.
\newblock A semigroup approach to wreath-product extensions of {S}olomon's
  descent algebras.
\newblock {\em Electron. J. Combin.}, 16(1):Research Paper 21, 9, 2009.

\bibitem{KapilovskyB}
G.~Karpilovsky.
\newblock {\em Group representations. {V}ol. 1. {P}art {B}}, volume 175 of {\em
  North-Holland Mathematics Studies}.
\newblock North-Holland Publishing Co., Amsterdam, 1992.
\newblock Introduction to group representations and characters.

\bibitem{Arbib}
K.~Krohn, J.~Rhodes, and B.~Tilson.
\newblock {\em Algebraic theory of machines, languages, and semigroups}.
\newblock Edited by Michael A. Arbib. With a major contribution by Kenneth
  Krohn and John L. Rhodes. Academic Press, New York, 1968.
\newblock Chapters 1, 5--9.

\bibitem{MantReut}
R.~Mantaci and C.~Reutenauer.
\newblock A generalization of {S}olomon's algebra for hyperoctahedral groups
  and other wreath products.
\newblock {\em Comm. Algebra}, 23(1):27--56, 1995.

\bibitem{MSS}
S.~Margolis, F.~Saliola, and B.~Steinberg.
\newblock Combinatorial topology and the global dimension of algebras arising
  in combinatorics.
\newblock {\em J. Eur. Math. Soc. (JEMS)}, 17(12):3037--3080, 2015.

\bibitem{ourmemoirs}
S.~Margolis, F.~V. Saliola, and B.~Steinberg.
\newblock Cell complexes, poset topology and the representation theory of
  algebras arising in algebraic combinatorics and discrete geometry.
\newblock {\em Mem. Amer. Math. Soc.}, 274(1345):xi+135, 2021.

\bibitem{rrbg}
S.~Margolis and B.~Steinberg.
\newblock The quiver of an algebra associated to the {M}antaci-{R}eutenauer
  descent algebra and the homology of regular semigroups.
\newblock {\em Algebr. Represent. Theory}, 14(1):131--159, 2011.

\bibitem{polo}
P.~Polo.
\newblock On {C}ohen-{M}acaulay posets, {K}oszul algebras and certain modules
  associated to {S}chubert varieties.
\newblock {\em Bull. London Math. Soc.}, 27(5):425--434, 1995.

\bibitem{qtheor}
J.~Rhodes and B.~Steinberg.
\newblock {\em The {$q$}-theory of finite semigroups}.
\newblock Springer Monographs in Mathematics. Springer, New York, 2009.

\bibitem{Rosenberg}
J.~Rosenberg.
\newblock {\em Algebraic {$K$}-theory and its applications}, volume 147 of {\em
  Graduate Texts in Mathematics}.
\newblock Springer-Verlag, New York, 1994.

\bibitem{Saliola}
F.~V. Saliola.
\newblock The quiver of the semigroup algebra of a left regular band.
\newblock {\em Internat. J. Algebra Comput.}, 17(8):1593--1610, 2007.

\bibitem{SaliolaDescent}
F.~V. Saliola.
\newblock On the quiver of the descent algebra.
\newblock {\em J. Algebra}, 320(11):3866--3894, 2008.

\bibitem{Saliolahyperplane}
F.~V. Saliola.
\newblock The face semigroup algebra of a hyperplane arrangement.
\newblock {\em Canad. J. Math.}, 61(4):904--929, 2009.

\bibitem{SolomonDescent}
L.~Solomon.
\newblock A {M}ackey formula in the group ring of a {C}oxeter group.
\newblock {\em J. Algebra}, 41(2):255--264, 1976.

\bibitem{repbook}
B.~Steinberg.
\newblock {\em Representation theory of finite monoids}.
\newblock Universitext. Springer, Cham, 2016.

\bibitem{affinetoprep}
B.~{Steinberg}.
\newblock {Topology and monoid representations I: Foundations}.
\newblock {\em arXiv e-prints}, April 2024.

\bibitem{woodcock}
D.~Woodcock.
\newblock Cohen-{M}acaulay complexes and {K}oszul rings.
\newblock {\em J. London Math. Soc. (2)}, 57(2):398--410, 1998.

\end{thebibliography}
\def\malce{\mathbin{\hbox{$\bigcirc$\rlap{\kern-7.75pt\raise0,50pt\hbox{${\tt
  m}$}}}}}\def\cprime{$'$} \def\cprime{$'$} \def\cprime{$'$} \def\cprime{$'$}
  \def\cprime{$'$} \def\cprime{$'$} \def\cprime{$'$} \def\cprime{$'$}
  \def\cprime{$'$} \def\cprime{$'$}

\end{document}